\documentclass[11pt,reqno]{amsart}
\usepackage[margin=1in]{geometry} 
\usepackage{amsfonts,amsmath,amsthm,amssymb,scrextend}
\usepackage{comment}

\usepackage{tikz}
\usepackage{tikz-cd}
\usepackage{pgfplots}
\usepackage{mathtools}
\usepackage[mathscr]{euscript}
 \let\mathscr\relax
\usepackage[scr]{rsfso}
\usepackage{multicol}
\usepackage{xcolor}
\usepackage[colorlinks=true, pdfstartview=FitV, linkcolor=blue,citecolor=blue, urlcolor=blue]{hyperref}
\usepackage{stmaryrd}
\usepackage{enumerate}
\usepackage[ruled,vlined]{algorithm2e}
\usepackage{algpseudocode}
\usepackage{tabu}

\DeclareMathOperator{\GL}{GL}

\DeclareMathOperator{\Id}{Id}
\DeclareMathOperator{\Aut}{Aut}
\DeclareMathOperator{\Inn}{Inn}
\DeclareMathOperator{\Out}{Out}
\DeclareMathOperator{\Ind}{Ind}
\DeclareMathOperator{\Res}{Res}

\DeclareMathOperator{\diag}{diag}
\DeclareMathOperator{\Conv}{Conv}
\DeclareMathOperator{\Spin}{Spin}
\DeclareMathOperator{\SL}{SL}
\let\O\undefine
\DeclareMathOperator{\O}{O}
\DeclareMathOperator{\U}{U}
\DeclareMathOperator{\Sp}{Sp}
\DeclareMathOperator{\SO}{SO}

\newtheorem{theorem}{Theorem}[section]
\newtheorem{lemma}[theorem]{Lemma}
\newtheorem{corollary}[theorem]{Corollary}

\newtheorem{proposition}[theorem]{Proposition}
\newtheorem{definition}[theorem]{Definition}

\newtheorem{remark}[theorem]{Remark}
\newtheorem{problem}[theorem]{Problem}

\newcommand{\im}{\mathrm{im}} 
\newcommand{\rank}{\mathrm{rank}} 
\newcommand{\RR}{\mathbb{R}} 
\newcommand{\CC}{\mathbb{C}} 
\newcommand{\QQ}{\mathbb{Q}} 
\newcommand{\FF}{\mathbb{F}} 
\newcommand{\ZZ}{\mathbb{Z}} 
\newcommand{\HH}{\mathbb{H}} 

\newcommand{\spn}{\mathrm{span}}
\newcommand{\Ad}{\mathrm{Ad}}
\newcommand{\ad}{\mathrm{ad}}
\newcommand{\W}{\mathcal{W}}
\newcommand{\tr}{\mathrm{tr}}
\newcommand{\tp}{{\scriptscriptstyle\mathsf{T}}}

\title{The Cayley transform on representations}
\author[J.~Y.~Lu]{Jingyu Lu}
\address{KLMM, Academy of Mathematics and Systems Science, Chinese Academy of Sciences, Beijing 100190, China}
\email{lujingyu@amss.ac.cn}
\author[K.~Ye]{Ke Ye }
\address{KLMM, Academy of Mathematics and Systems Science, Chinese Academy of Sciences, Beijing 100190, China}
\email{keyk@amss.ac.cn}

\begin{document}
%
\begin{abstract}
The classical Cayley transform is a birational map between a quadratic matrix group and its Lie algebra,  which was first discovered by Cayley in 1846.  Because of its essential role in both pure and applied mathematics,  the classical Cayley transform has been generalized from various perspectives.  This paper is concerned with a representation theoretic generalization of the classical Cayley transform.  The idea underlying this work is that the applicability of the classical Cayley transform heavily depends on how the Lie group is represented.  The goal is to characterize irreducible representations of a Lie group,  to which the classical Cayley transform is applicable.  To this end,  we first establish criteria of the applicability for a general Lie group.  If the group is semisimple,  we further obtain a geometric condition on the weight diagram of such representations.  Lastly,  we provide a complete classification for complex simple  Lie groups and their compact real forms.  Except for the previously known examples,  spin representations of $\Spin(8)$ are the only ones on our list.
\end{abstract}
\maketitle
\section{Introduction}
\subsection*{Classical Cayley transform} Let $\mathbb{F}$ be either $\mathbb{C}$ or $\mathbb{R}$.  The \emph{classical Cayley transform} is defined as the birational map 
\begin{equation}\label{eq:classicalCayley}
C: \mathfrak{gl}_n(\mathbb{F})  \dasharrow \GL_n(\mathbb{F}),\quad C(u) \coloneqq (I_n + u)(I_n  - u)^{-1}.
\end{equation}
The classical Cayley transform is \emph{applicable} to a Lie subgroup $G \subseteq \GL_n(\mathbb{F})$ if $C$ is defined on a neighbourhood $V$ of $0$ in $\mathfrak{g}$ such that $C(V) \subseteq G$,  where $\mathfrak{g}$ is the Lie algebra of $G$.  Familiar examples of such groups are:
\begin{itemize}
\item[$\diamond$] \emph{Quadratic matrix group} \cite{Cayley1846,  Weyl39,  Weil60,Postnikov86}: Let $\mathbb{V}$ be an $\mathbb{F}$-vector space and let $\sigma: \GL(\mathbb{V}) \to \GL(\mathbb{V})$ be an $\mathbb{R}$-linear involution.  For any $B\in \GL(\mathbb{V})$ such that $\sigma(B) = \pm B$,  the classical Cayley transform is applicable to $G_B \coloneqq  \left\{g \in   \GL(\mathbb{V}): B(gx,gy) = B(x,y), \; x, y \in \mathbb{V} \right\}$.  By choosing different $\sigma$ and $B$,  we obtain classical matrix groups $\O_n(\mathbb{F}),  \Sp_{2n}(\mathbb{F})$ and $\U_n$ as special examples.
\item[$\diamond$] \emph{Diagonal matrix group} \cite[Example~1.20]{LPR06}: $C$ is applicable to $T_n$  where $T_n \subseteq \GL_n(\mathbb{F})$ is the group consisting of all invertible diagonal matrices.  
\item[$\diamond$] \emph{Unipotent group} \cite{BR85,KM03}: Let $R_n \subseteq \GL_n(\mathbb{F})$ be the subgroup of all upper triangular unipotent matrices.  $C$ is applicable to $R_n$.  In fact,  we have $I_n + \mathfrak{r}_n = R_n$ from which we obtain $C(\mathfrak{r}_n) \subseteq  I_n + \mathfrak{r}_n = R_n$.  Here $\mathfrak{r}_n$ denotes the Lie algebra of $R_n$,  consisting of strictly upper triangular matrices.
\item[$\diamond$] \emph{Upper triangular matrix group}: Let $B_n \subseteq \GL_n(\mathbb{F})$ be the group of all invertible upper triangular matrices.  Then it is straightforward to verify that $C$ is applicable to $B_n$. 
\end{itemize}

If $C$ is applicable to $G \subseteq \GL_n(\mathbb{F})$,  then we obtain a \emph{simple and explicit isomorphism} between neighbourhoods of $0 \in \mathfrak{g}$ and $I_n \in G$.  This simple observation leads to numerous applications of the classical Cayley transform in both pure and applied mathematics.  For instance,  the LS category of quadratic matrix groups can be computed by the Cayley transform \cite{GMP11}; The Cayley transform is closely related to several topological and geometric invariants of super vector bundles: sub-bundles,  superconnection character form and Bott maps \cite{Quillen85, Quillen88}; Symmetric Siegel domains may be characterized by the geometry of images of the Cayley transform \cite{Nomura01,  Nomura03, Kai07}.  On the other side,  $C(u/2)$ is the best Pad\'{e} approximant of type $(1,1)$ of the matrix exponential function $\exp(u)$ \cite{BG96}.  In particular,  $C(u/2)$ is a second order approximation of $\exp(u)$.  This is a desirable property for numerical computations in various fields of applied mathematics such as numerical linear algebra \cite{MR96,BCM04, MST24},  numerical differential equation and integration \cite{Iserles01,PX01,HLW10, WGM21} and Riemannian optimization \cite{WY13, GSAS21}.  Moreover,  the Cayley transform is widely employed by practitioners in statistics \cite{JHD20},  machine learning \cite{HWY18,  ESAH20} and signal processing \cite{JH03, MD24}. 

The above mentioned applications of the Cayley transform is a consequence of the fact that it establishes a correspondence between a Lie group and its Lie algebra.  It is noticeable that the Cayley transform can also be defined on the complex plane or some space of operators,  which makes it a fundamental ingredient in complex analysis \cite{Krantz99,ST20,vWW24} and operator theory \cite{SF70,  Lance95, JM17,BKR20}. 
\subsection*{Generalized Cayley transform}
Because of its essential role in both theoretical studies and practical applications,  there have been several attempts to generalize the Cayley transform in the past 40 years.  The first generalization is given by an analytic construction based on the Iwasawa decomposition of H-type groups \cite{Helgason78,CDK91,CDKR98,ACD04,AD06},  but it is no longer a map between the Lie group and its Lie algebra.  The second generalization is for homogeneous spaces of quadratic matrix groups \cite{GMP11,GSAS21},  which is specifically designed for efficient calculations on these quotient manifolds.  The third generalization is for algebraic groups,  which defines a \emph{generalized Cayley transform} as an equivariant birational isomorphism between an algebraic group and its Lie algebra \cite{BR85,KM03}.  The ultimate goal of all existing works \cite{KM03,LPR06,BKLR14,BK15,Borovoi15} is to solve the following problem,  which is a natural extension of Luna's problem \cite{Domingo75} for $\SL_n$.
\begin{problem}[Lemire-Popov-Reichstein problem]\cite[Problem~1.7]{LPR06}\label{prob:LPR}
Which linear algebraic group can have a generalized Cayley transform?
\end{problem}
\subsection*{Main results} The purpose of this paper is to investigate a representation theoretic generalization of the classical Cayley transform.  We notice that the applicability of the classical Cayley transform $C$ depends on how we represent the group.  For instance,  it is straightforward to verify that $C$ is applicable to $\SL_2(\mathbb{C}) \subseteq \mathbb{C}^{2 \times 2}$.  It is also applicable to $\Ad(\SL_2(\mathbb{C})) \subseteq \GL(\mathfrak{sl}_2(\mathbb{C})) \simeq \GL_3(\mathbb{C})$ where $\Ad: \SL_2(\mathbb{C}) \to \GL(\mathfrak{sl}_2(\mathbb{C}))$ is the adjoint representation of $\SL_2(\mathbb{C})$.  In fact,  $\SL_2(\mathbb{C})$ is a double cover of $\SO_3(\mathbb{C})$ and $\Ad(\SL_2(\mathbb{C})) = \SO_3(\mathbb{C})$ if we identify $\mathfrak{sl}_2(\mathbb{C})$ with $\mathbb{C}^3$.  However,  $C$ is not applicable to $\rho(\SL_2(\mathbb{C})) \subseteq \GL\left( \mathsf{S}^3\mathbb{C}^2 \right) \simeq \GL_4(\mathbb{C})$ where $\rho: \SL_2(\mathbb{C}) \to \GL(\mathsf{S}^3 \mathbb{C}^2)$ is the irreducible representation of $\SL_2(\mathbb{C})$ on the third symmetric power $\mathsf{S}^3(\mathbb{C}^2)$.  Because of this observation,  we have the following analogue of the Lemire-Popov-Reichstein (abbreviated as LPR) problem~\ref{prob:LPR}.
\begin{problem}[Representation theoretic LPR problem]
\label{prob:LPR-rep}
To which pair $(G,\rho)$ of a Lie group $G$ and its representation $\rho: G\to \GL(\mathbb{V})$,  the classical Cayley transform $C$ defined in \eqref{eq:classicalCayley} is applicable?
\end{problem}
Here by Definition~\ref{def:Cayley representation},  $C$ is \emph{applicable} to $\rho$ if there is an open neighbourhood $V$ of $0$ in $\mathfrak{g}$ such that $C \circ d\rho (V) \subseteq \rho (G)$.  We notice that if $G$ is a linear algebraic group and $\rho: G \to \GL(\mathbb{V})$ is a faithful representation to which $C$ is applicable,  then clearly $G$ satisfies the requirement in the LPR problem.  This paper is devoted to answering Problem~\ref{prob:LPR-rep}.  First,  we establish in Theorem~\ref{cond:3} the equivalence between the applicability of $C$ to $\rho$ and the power span property of  $d\rho (\mathfrak{g})$ (cf. Definition~\ref{cond:span}).  Next,  we restrict our discussion to semisimple Lie groups in Section~\ref{sec:semisimple} and obtain the equivalence between the applicability of $C$ to $\rho$ and the power span property of $d\rho (\mathfrak{h})$ in Theorem~\ref{cond:simple}.  Here $\mathfrak{h}$ denotes the Cartan subalgebra of $\mathfrak{g}$.  Moreover,  we prove in Theorem~\ref{thm:main} that if $\rho$ is irreducible and $C$ is applicable to $\rho$,  then the weight diagram of $\rho$ must be the orbit of the highest weight under the action of Weyl group,  possibly union with the origin.  Lastly,  Section~\ref{sec:simple} is concerned with a complete classification of irreducible representations of complex simple Lie groups and their compact real forms,  to which $C$ is applicable.  Our classification results (cf. Theorems~\ref{thm:classification1} and \ref{thm:classification2}) indicate that the only irreducible representations to which $C$ is applicable are: 
\begin{itemize} 
\item[$\diamond$] $G$ is of type $A_1$,  $B_n$ ($n\ge 1$),  $C_n$ ($n \ge 3$) or $D_n$ ($n \ge 3$) and $\rho$ is the standard representation.
\item[$\diamond$]  $G$ is of type $D_4$ and $\rho$ is one of the two spin representations.
\end{itemize}
Our results may be recognized as representation theoretic analogues of \cite[Theorem 1.31]{LPR06}. 

\subsection*{Caveats}
In this paper, we consider Lie groups and Lie algebras over $\mathbb{R}$ and $\mathbb{C}$.  Unless otherwise stated, representations to be discussed are all finite-dimensional complex vector spaces (except for the adjoint representations of Lie groups and Lie algebras),  even for real Lie groups and real Lie algebras.  A morphism between two real (resp.  complex) Lie groups is a smooth (resp.  complex analytic) group homomorphism.

\section{Preliminaries and notations}
We use capital letters such as $G$,  $H$,  $K$ and $U$ to denote Lie groups and their Lie algebras are correspondingly denoted by $\mathfrak{g}$,  $\mathfrak{h}$,  $\mathfrak{k} $ and $\mathfrak{u}$.  In particular,  $\mathfrak{h}$ is reserved for Cartan subalgebra.  We denote by $\mathcal{U}(\mathfrak{g})$ the universal enveloping algebra of $\mathfrak{g}$.

\subsection{Real and complex semisimple Lie groups}

\subsubsection{Complexification}
Let $\mathfrak{g}$ be a real semisimple Lie algebra.  We denote its \emph{complexification} by $\mathfrak{g}^{\mathbb{C}} \coloneqq \mathfrak{g} \otimes \mathbb{C}$.  A \emph{complexification of a real Lie group $G$} \cite[Chapitre 3]{bourbaki2006groupes} is a complex Lie group $G^{\CC}$ together with a real Lie group morphism $\tau: G \to G^{\CC}$ satisfying the following universal property: For any complex Lie group $H$ and any real Lie group morphism $\psi: G \to H$,  there exists a unique complex Lie group morphism $\varphi: G^{\CC} \to H$ such that $\varphi \circ \tau = \psi$. The complexification of a connected real Lie group $G$ with Lie algebra $\mathfrak{g}$ always exists.  If moreover $G$ is simply connected,  then the Lie algebra of $G^\CC$ is isomorphic to $\mathfrak{g}^\CC$ \cite[Chapitre 3,4,5,6]{bourbaki2006groupes}.

\subsubsection{Real form}
A \emph{real form of a complex semisimple Lie algebra $\mathfrak{u}$} \cite[Lecture 26]{fulton2013representation} is a real Lie subalgebra $\mathfrak{u}'$ of $\mathfrak{u}$ such that ${\mathfrak{u}'}^{\mathbb{C}}\simeq \mathfrak{u}$.  Let $U$ be a connected complex Lie group.  A \emph{real form of $U$} is a real Lie subgroup $U'$ of $U$ such that $U$ is a complexification of $U'$ and $\mathfrak{u}'^{\CC} \simeq \mathfrak{u}$. Real forms of a complex Lie group are not unique in general.  In Section \ref{sec:semisimple}, we will discuss two important real forms of complex semisimple Lie groups: the split form and the compact form.

\subsubsection{Restricted and induced representation}
Given a complex semisimple Lie algebra $\mathfrak{u}$ and a representation $\pi: \mathfrak{u} \to \mathfrak{gl}(\mathbb{V})$,  the restricted representation of $\pi$ to its real form $\mathfrak{u}_0$ is defined by 
\[
\Res(\pi) \coloneqq \pi|_{\mathfrak{u}_0}: \mathfrak{u}_0 \to \mathfrak{gl}(\mathbb{V}),\quad 
\Res(\pi)(x) = \pi(x).
\]
Similarly,  given a complex semisimple Lie group $U$ and a representation $\rho: U \to \GL(\mathbb{V})$,  the \emph{restricted representation} of $\rho$ to a real form $U_0$ of $U$ is  
\[
\Res(\rho) \coloneqq \rho|_{U_0}: U_0 \to \GL(\mathbb{V}),\quad \Res(\rho)(g) = \rho(g).
\]

Suppose that $\pi: \mathfrak{g} \to \mathfrak{gl}(\mathbb{V})$ is a representation of a real semisimple Lie algebra $\mathfrak{g}$.  The \emph{induced representation} of $\pi$ is 
\[
    \Ind(\pi): \mathfrak{g}^{\mathbb{C}} \to \mathfrak{gl}(\mathbb{V}),\quad \Ind(\pi)(x_1 + i x_2) = \pi(x_1) + i \pi(x_2), 
\]
where $x_1,x_2 \in \mathfrak{g}$.  Since $\mathbb{V}$ is a complex vectors space,  $\Ind(\pi)$ is well-defined.  If $\rho: G \to \GL(\mathbb{V})$ is a representation of a real semisimple Lie group $G$ and $(G^{\mathbb{C}},\tau)$ is its complexification.  The \emph{induced representation} of $\rho$ is the representation $\Ind (\rho):G^\mathbb{C} \to \GL(\mathbb{V})$ uniquely determined by the universal property of $(G^\mathbb{C},\tau)$. If moreover $G$ is a real form of $G^{\CC}$, then explicitly we have
\[
\Ind(\rho): G^\mathbb{C} \to \GL(\mathbb{V}),\quad \Ind(\rho)(\exp (x)) = \exp (\Ind( d \rho) (x)),
\]
where $x\in \mathfrak{g}^{\mathbb{C}}$ and $\exp:\mathfrak{g} \to G$ is the exponential map.

\subsubsection{Weight lattice and irreducible representation}
Let $\mathfrak{g}$ be a semisimple Lie algebra.  Suppose that $\Lambda$ (resp.  $\Phi$,  $\W$,  $C$) is the weight lattice (resp.  set of roots,  Weyl group,  fundamental Weyl chamber) of $\mathfrak{g}$.  
\begin{theorem}\cite[Theorem 14.18]{fulton2013representation}\label{thm: FTRTSLA}
We have the following:
\begin{enumerate}[(a)]
\item There is a one to one correspondence between $C \cap \Lambda$ and the set of isomorphism classes of irreducible finite dimensional representations of $\mathfrak{g}$.  

\item Let $\rho: \mathfrak{g} \to \mathfrak{gl}(\mathbb{V})$ be the irreducible finite dimensional representation with highest weight $\omega$ and let $\mathbb{V} = \bigoplus_{\lambda \in \mathbb{V}} \mathbb{V}_{\lambda}$ be the weight space decomposition.  Then
$\{\lambda \in \Lambda: \mathbb{V}_{\lambda} \ne 0 \} = \Conv(\W \omega) \cap \left( \omega + \ZZ\Phi \right)$,  where $\W \omega$ is the $\W$-orbit of $\omega$ and $\Conv(\W \omega)$ denotes the convex hull of $\W \omega$. 
\end{enumerate}
\end{theorem}  
\subsection{Symmetric algebraic groups}
A complex linear algebraic group $G \subseteq \GL_n(\mathbb{C})$ is called \emph{linearly reductive} \cite[Section 3.1.3]{wallach2017geometric} if all the finite-dimensional regular representations are completely reducible . A complex linear algebraic group $G \subseteq \GL_n(\mathbb{C})$ is called \emph{symmetric}  \cite[Section 3.1.3]{wallach2017geometric}  if $G^\ast =G$, where $G^\ast \coloneqq \left\{\overline{g}^\tp: g \in G \right\}$. It is worth remarking that linear reductivity is independent of embeddings,  whereas the definition of a symmetric group depends on the specific embedding of $G$ into $\GL_n(\mathbb{C})$.   Lemma \ref{lem: linear reductive sym} below reveals that linearly reductive groups and symmetric groups are two sides of the same coin.

\begin{lemma}\label{lem: linear reductive sym}\cite[Lemma 3.5, Theorem 3.13]{wallach2017geometric}
Let $G \subseteq \GL_n(\mathbb{C})$ be a  linear algebraic group.  Then $G$ is linearly reductive if and only if there is some $P \in \GL_n(\mathbb{C})$ such that $P G P^{-1}$ is a  symmetric group.  
\end{lemma}

\begin{theorem}[Polar decomposition]\label{thm: polar decomposition}\cite[Theorem 2.12]{wallach2017geometric}
Let $G \subseteq \GL_n(\CC)$ be a symmetric algebraic group and let $K \coloneqq G\cap U_n$ where $U_n$ is the group of $n\times n$  unitary matrices.  We denote by $\mathfrak{k}$ the Lie algebra of $K$.  Then the map $p: K \times i \mathfrak{k} \to G$ defined by $p(k,  Z) \coloneqq k \exp (Z)$ is a diffeomorphism.  In particular,  $K$ is connected if and only if $G$ is an irreducible algebraic variety.
\end{theorem}

A subgroup $H$ of an algebraic group $G$ is called an \emph{irreducible algebraic subgroup} if $H$ is an irreducible subvariety of $G$.
\begin{proposition}\cite[Section 4.1.2]{onishchik2012lie}\label{prop:semisimple algebraic}
For each semisimple Lie subalgebra $\mathfrak{g} \subseteq \mathfrak{gl}_n(\mathbb{F})$,  there exists an irreducible algebraic subgroup $G \subseteq \GL_n(\mathbb{F})$ such that $\mathfrak{g}$ is the Lie algebra of $G$.  Moreover,  if $\mathbb{F} = \mathbb{C}$,  then $G$ can be chosen to be connected in the Euclidean topology. 
\end{proposition}

We remark that the equivalence between the connectedness of $K$ and $G$ follows immediately from the diffeomorphism $K \times i \mathfrak{k} \simeq G$.  By \cite[Corollary 2.14]{wallach2017geometric} ,  the connectedness of $G$ is further equivalent to the irreducibility of $G$ as an algebraic variety.  
\begin{theorem}[Cartan decomposition]\label{thm: Cartan decomposition}\cite[Corollary 2.20]{wallach2017geometric}
Let $G$ be a connected symmetric subgroup of $\GL_n(\CC)$ and let $K$ = $G\cap U_n$ be as above. Then we have $\mathfrak{g} = \mathfrak{k} \bigoplus_{\RR} i\mathfrak{k}$.  Moreover,  If $T$ is a maximal compact torus of $K$ with Lie algebra $\mathfrak{t}$,  then $\mathfrak{t}$ is a Cartan algebra of $K$ and $ \mathfrak{k} = \Ad(K)\mathfrak{t}$.
\end{theorem}
\section{Some basic facts}
In this short section,  we establish two basic facts that will be needed in the sequel.  The first fact is about the relation between the induction and the restriction of representations.  
\begin{proposition}\label{prop: irr R,C rep}
Let $\mathfrak{g}$ (resp.  $\mathfrak{u}$) be a real (resp.  complex) semisimple Lie algebra and let $\mathfrak{g}^{\mathbb{C}}$ (resp.  $\mathfrak{u}'$) be its complexification (resp.  real form).  Suppose that $\pi: \mathfrak{g} \to \mathfrak{gl}(\mathbb{V})$ (resp.  $\rho: \mathfrak{u} \to \mathfrak{gl}(\mathbb{W})$) is a representation of $\mathfrak{g}$ (resp.  $\mathfrak{u}$).  Then we have the followings properties:
\begin{enumerate}[(a)]
\item $\Res ( \Ind(\pi) ) = \pi$ and $\Ind (\Res (\rho)) = \rho$. \label{prop: irr R,C rep:item1}
\item $\pi$ (resp.  $\rho$) is irreducible if and only if $ \Ind(\pi)$ (resp.  $\Res(\rho)$) is irreducible.\label{prop: irr R,C rep:item2}
\item There are irreducible representations $\pi_j: \mathfrak{g} \to \mathfrak{gl}(\mathbb{V}_j)$,  $1 \le j \le s$,   such that 
\[
\mathbb{V} = \bigoplus_{j=1}^s \mathbb{V}_j,\quad \pi = \prod_{j=1}^s \pi_j.
\]\label{prop: irr R,C rep:item3}
\end{enumerate}
\end{proposition}
\begin{proof}
Property \eqref{prop: irr R,C rep:item1} follows immediately from the definition.  Suppose that $\pi$ is irreducible and $0 \ne \mathbb{V}' \subseteq \mathbb{V}$ is a subspace such that $\Ind(\pi)(\mathfrak{g})(\mathbb{V}’) \subseteq \mathbb{V}$.  Then we have 
\[
\mathbb{V\subseteq} \pi (\mathfrak{g}) (\mathbb{V}') = \Res (\Ind (\pi)) (\mathfrak{g}) ) (\mathbb{V}’) = \Ind(\pi) (\mathfrak{g}) (\mathbb{V}’) \subseteq \mathbb{V}',
\]
which implies that $\Ind(\pi)$ is irreducible.  Taking $\pi = \Res(\rho)$,  we obtain that the irreducibility of $\Res(\rho)$ leads to the irreducibility of $\rho = \Ind (\Res(\rho))$.  

Conversely,  if $\rho$ is irreducible and $0 \ne \mathbb{W}' \subseteq \mathbb{W}$ is a subspace such that $\Res(\rho) (\mathfrak{u}') (\mathbb{W}') \subseteq \mathbb{W}'$,  then 
\[
\rho (\mathfrak{u}) (\mathbb{W}') = \Ind (\Res(\rho)) \left( {\mathfrak{u}'}^\mathbb{C}\right) \left( \mathbb{W}'\right) =  \Res(\rho)(\mathfrak{u}')(\mathbb{W}')  + i  \Res(\rho)(\mathfrak{u}') (\mathbb{W}') \subseteq \mathbb{W}'. 
\]
Thus,  we obtain the irreducibility of $\Res(\rho)$.  Taking $\rho = \Ind(\pi)$,  we conclude that the irreducibility of $\Ind(\pi)$ implies the irreducibility of $\pi = \Res(\Ind(\pi))$ and this completes the proof of \eqref{prop: irr R,C rep:item2}.

To prove \eqref{prop: irr R,C rep:item3},  we notice that there are irreducible representations $\rho_j: \mathfrak{g} \to \mathfrak{gl}(\mathbb{V}_j)$,  $1 \le j \le s$,  such that 
\[
\mathbb{V} = \bigoplus_{j=1}^s \mathbb{V}_j,\quad \Ind(\pi) = \prod_{j=1}^s \rho_j.
\]
For each $1 \le j \le s$,  we define $\pi_j \coloneqq \Res(\rho_j)$,  which is irreducible according to  \eqref{prop: irr R,C rep:item2}.  It is clear that $\pi = \Res(\Ind(\pi)) = \prod_{j=1}^s \Res(\rho_j)$ and this proves \eqref{prop: irr R,C rep:item3}.
\end{proof}

The second fact we will need is the generation of a semisimple Lie algebra by its Cartan subalgebra.  To achieve this,  we first establish two lemmas.
\begin{lemma}\label{cor: semisimple is symmetric}
Let $\mathfrak{g}$ be a complex semisimple Lie algebra and let $\ad: \mathfrak{g} \to \mathfrak{gl}(\mathfrak{g})$ be its adjoint representation.  There exists an irreducible linearly reductive algebraic group $G \subseteq \GL(\mathfrak{g}) $ whose Lie algebra is isomorphic to $\ad(\mathfrak{g}) \simeq \mathfrak{g}$. By a change of coordinates of $\mathfrak{g}$, one can furthermore require $G$ to be symmetric.  
\end{lemma}
\begin{proof}
The semisimplicity of $\mathfrak{g}$ ensures that $\ad$ is faithful.  Hence we have $\mathfrak{g} \simeq \ad(\mathfrak{g}) \subseteq \mathfrak{gl}(\mathfrak{g})$. According to Proposition \ref{prop:semisimple algebraic},  there is an irreducible algebraic subgroup $G' \subseteq \GL(\mathfrak{g})$ whose Lie algebra is equal to $\ad(\mathfrak{g})$.  We claim that $G'$ is linearly reductive.  Then Lemma \ref{lem: linear reductive sym} implies the existence of $P \in \GL(\mathfrak{g})$ such that $G \coloneqq P G' P^{-1} $ is the desired group.  Therefore,  it is left to prove the claim.  Let $\rho: G \to \GL(\mathbb{V})$ be a finite dimensional rational representation of $G$.  Since the Lie algebra of $G$ is isomorphic to $\mathfrak{g}$,  $\rho$ induces a representation $\pi: \mathfrak{g}  \to \mathfrak{gl}(\mathbb{V})$.  The semisimplicity of $\mathfrak{g}$ implies $\pi = \prod_{j=1}^m \pi_j: \mathfrak{g} \to \bigoplus_{j=1}^m \mathfrak{gl}(\mathbb{V}_j)$ where $\pi_j:  \mathfrak{g} \to \mathfrak{gl}(\mathbb{V}_j)$ is an irreducible representation for each $1 \le j \le m$.  We recall that $\pi$ is induced by $\rho$,  thus $\rho$ admits a decomposition $\rho = \prod_{j=1}^m \rho_j:  G \to \prod_{j=1}^m \GL(\mathbb{V}_j)$ where $\rho_j: G \to \GL(\mathbb{V}_j)$ is a representation inducing $\pi_j$.  Moreover,  by construction $\rho_j$ is irreducible for each $1 \le j \le m$.

By choosing a basis for $\mathfrak{g}$,  we may identify $\GL(\mathfrak{g})$ with $\GL_n(\mathbb{C})$,  where $n = \dim \mathfrak{g}$.  If $G \subseteq \GL(\mathfrak{g})$ is linearly reductive,  then Lemma~\ref{lem: linear reductive sym} implies that $P G P^{-1} \subseteq  \GL(\mathfrak{g})$ is symmetric for some $P\in  \GL(\mathfrak{g})$.  As a consequence,  the change of coordinates of $\mathfrak{g}$ determined by $P$ renders $G$ a symmetric group.
\end{proof}

\begin{lemma}\label{thm: adjoint factor thru}
Let $G$ be a connected Lie group with Lie algebra $\mathfrak{g}$.  Suppose $p:  \widetilde{G} \to G$ is the universal covering of $G$ and $q: G \to G_0 \coloneqq G/Z(G)$ is the quotient map where $Z(G)$ is the center of $G$.  Then we have the following commutative diagram: 
\[
\begin{tikzcd}
\widetilde{G} \arrow{r}{\widetilde{\Ad}} \arrow[swap]{d}{p}
& \GL(\mathfrak{g}) \\
G \arrow[swap]{r}{q} & G_0 \arrow[swap]{u}{\Ad_0}
\end{tikzcd}        
\]
where $\widetilde{\Ad}$ (resp.  $\Ad_0$) is the adjoint representation of $\widetilde{G}$ (resp.  $G_0$).
\end{lemma}
\begin{proof}
{Notice that $\ker \left( \widetilde{\Ad} \right) = Z \left( \widetilde{G} \right)$,  $\ker(p) \subseteq Z \left( \widetilde{G} \right)$ and $p\left( Z\left( \widetilde{G} \right) \right) \subseteq Z(G) = \ker(q)$.  We may conclude that $\ker(q \circ p) = p^{-1} (Z(G)) = Z\left( \widetilde{G} \right)$. The commutativity of the diagram follows immediately. }
\end{proof}
We conclude this section by deriving the generation of a semisimple Lie algebra by its Cartan subalgebra.
{\begin{proposition}\label{prop: compact form generates}
Let $G$ be a compact semisimple Lie group with Lie algebra $\mathfrak{g}$.  We fix a root system for $G$ and denote by $\mathfrak{h}$ the Cartan subalgebra of $\mathfrak{g}$.  Then we have $\mathfrak{g} = \Ad(G)(\mathfrak{h})$.
\end{proposition}
\begin{proof}
Without loss of generality,  we assume that $G$ is connected.  Since any two Cartan subalgebras of $\mathfrak{g}$ are conjugated by the action of $G$, we are free to choose the one stated in Theorem \ref{thm: Cartan decomposition}.  By Lemma~\ref{cor: semisimple is symmetric},  there is a symmetric irreducible algebraic subgroup $G' \subseteq \GL \left( \mathfrak{g}^{\mathbb{C}} \right)$ whose Lie algebra is isomorphic to $\mathfrak{g}^{\mathbb{C}}$.  Let $U \subseteq \GL \left( \mathfrak{g}^{\mathbb{C}} \right)$ be the subgroup consisting of unitary elements and let $K \coloneqq G' \cap U$.  By Theorem \ref{thm: polar decomposition},  $K$ is connected and $\mathfrak{k}^{\mathbb{C}} \simeq \mathfrak{g}^{\mathbb{C}}$.  The uniqueness of the compact form of a semisimple Lie algebra implies $\mathfrak{k} \simeq \mathfrak{g}$.  Moreover,  by Theorem~\ref{thm: Cartan decomposition} we have $\mathfrak{k} = \Ad(K)(\mathfrak{h}_K)$ where $\mathfrak{h}_K$ is the Cartan subalgebra of $\mathfrak{k}$.  Under the isomorphism $\mathfrak{k} \simeq \mathfrak{g}$,   Lemma~\ref{thm: adjoint factor thru} implies $\mathfrak{g} = \Ad(G) \mathfrak{h}$.
\end{proof}}

\section{The Cayley transform on a representation}\label{sec:Cayley}
Suppose $\mathbb{F}$ is either $\mathbb{R}$ or $\mathbb{C}$.  Let $G$ be a Lie group and let $\rho: G \to   \GL(\mathbb{V})$ be a representation of $G$.  Here $\mathbb{V}$ is a finite-dimensional complex vector space. We denote by $\mathfrak{g}$ the Lie algebra of $G$. Assume $C: \mathfrak{gl}(\mathbb{V}) \dasharrow \GL(\mathbb{V})$ is the classical Cayley transform defined by $C(u) = (1 + u)/(1 - u)$.
\begin{definition}[Cayley representation]\label{def:Cayley representation}
We say that $\rho$ is a Cayley representation if there exist a neighbourhood $V_1\subseteq \mathfrak{g}$ of $0$ and a neighbourhood $V_2 \subseteq \im \rho$\footnote{The group $\im \rho$ is called the \emph{immersive Lie subgroup} \cite[Definition 4.2]{chenweihuan2001}.} of the identity element such that $C$ is defined on $d\rho (V_1)$ and $C(d\rho(V)) \subseteq \rho(G)$.  We also say that the Cayley transform is applicable to $\rho$.
\end{definition}
Clearly,  given a matrix Lie group $G \subseteq \GL_n(\mathbb{\CC})$,  the natural inclusion $\iota: G \hookrightarrow \GL_n(\mathbb{\CC})$ is a Cayley representation if and only if $C$ is applicable to $G$ in the classical sense.  The following basic properties are well-known for the classical Cayley transform.
\begin{proposition}[Basic properties]\label{prop:basicprop}
Let $\rho:  G \to \GL(\mathbb{V})$ be a representation to which the Cayley transform is applicable.  For any $X \in \mathfrak{g}$ such that $C$ is defined at $d\rho (X)$,  we have 
\begin{enumerate}[(a)]
\item $C(-d\rho(X)) C(d\rho(X)) = \Id_{\mathbb{V}}$.  \label{prop:basicprop:itema}
\item Denote $f(X) \coloneqq C ( d\rho (X)/2 )$ and $g(X) \coloneqq \exp (d\rho(X))$.  Then $f$ is a second order approximation of $g$ near the origin.  \label{prop:basicprop:itemb}
\item If $\pi: G \to  \GL(\mathbb{W})$ is Cayley,  then $\rho \times \pi: G \to \GL(\mathbb{V}) \times \GL(\mathbb{W})$ is Cayley. \label{prop:basicprop:itemc}
\end{enumerate}
\end{proposition}
\begin{proof}
One can verify \eqref{prop:basicprop:itema} and \eqref{prop:basicprop:itemc} by definition.  If we denote $u \coloneqq d\rho(X)$,  then \eqref{prop:basicprop:itemb} follows from
\[
f(X) - g(X) = C\left( \frac{u}{2} \right) - \exp(u) =  1 +u  \sum_{k=0}^{\infty}\frac{u^k}{2^k} - \sum_{k =0}^{\infty}\frac{u^k}{k!}= O(X^3).  \qedhere
\]
\end{proof}

We remark that $\rho(G)$ is a subgroup of $\GL(\mathbb{V})$,  but it is not necessarily a Lie subgroup.  For instance,  we consider the representation of $\mathbb{R}$ on $\mathbb{C}^2$ defined by 
\[
\rho_\alpha: \mathbb{R} \to \GL(\mathbb{C}^2),\quad \rho_{\alpha}(t) = (\exp(it), \exp(i \alpha t)),
\]
where $\alpha$ is a fixed real number.  Clearly,  $\rho_{\alpha}(\mathbb{R})$ is a Lie subgroup of $\GL(\mathbb{R}^2)$ if and only if $\alpha \in \QQ$.  However,  it is clear from the definition that $d\rho (\mathfrak{g})$ is a Lie subalgebra of $\mathfrak{gl}(\mathbb{V})$.  The following proposition indicates that the applicability of the Cayley transform only depends on $\mathfrak{g}$ and its representation $d\rho: \mathfrak{g} \to \mathfrak{gl}(\mathbb{V})$.  

\begin{proposition}[Independence]\label{prop: group to Lie algebra}
Let $G$ be a Lie group and let $\pi: G_0 \to G$ be the universal cover of $G$.  A representation $\rho: G\to \GL(\mathbb{V})$ is Cayley if and only if $\rho \circ \pi: G_0 \to \GL(\mathbb{V})$ is Cayley.  
\end{proposition}
\begin{proof}
The 'only if' part follows immediately from the definition.  We notice that $G_0$ and $G$ have the same Lie algebra $\mathfrak{g}$.  If $\rho \circ \pi$ is Cayley,  then we have $C ( d (\rho \circ \pi) (V) ) \subseteq (\rho \circ \pi) (G_0) $ from some open neighbourhood $V$ of $0 \in \mathfrak{g}$.  Since $\pi$ is a covering map,  we have $d\pi  = \Id_{\mathfrak{g}}$ and $\pi(G_0) = G$.  This implies $C (d \rho (V)) \subseteq \rho(G)$.
\end{proof}

We fix a norm on $\mathbb{V}$ and denote by $\lVert x \rVert$ the corresponding operator norm of $x \in \mathfrak{gl}(\mathbb{V})$.

\begin{lemma} \label{crit:1}
For any $u \in  d\rho (\mathfrak{g})$ with $\lVert u \rVert <  \frac{1}{3}$,  $C(u)$ is defined.  Moreover,  we have $C(u) \in \rho(G)$ if and only if  $\log(1+u)-\log(1-u) \in  d\rho (\mathfrak{g})$.
\end{lemma}
\begin{proof}
We recall that $\exp: \mathfrak{gl}(\mathbb{V}) \to \GL(\mathbb{V})$ is a homeomorphism between a neighbourhood $U$ of $0 \in \mathfrak{gl}(\mathbb{V})$ and a neighbourhood $V$ of $\Id \in \GL(\mathbb{V})$.  Conversely,  the map $\log: \GL(\mathbb{V}) \dasharrow \mathfrak{gl}(\mathbb{V})$ is defined for $a \in \GL(\mathbb{V})$ such that $\lVert a - \Id \rVert < 1$,  and it is the inverse map of $\exp$ whenever its defined.  Hence we may choose $V = \{a\in \GL(\mathbb{V}): \lVert a - \Id \rVert <1 \}$ and $U = \log (V)$.  

Since $\| u \| < \frac{1}{3}$,  $C(u)$ is defined by definition.  Moreover,  we have $\left\|  C(u) - \Id \right\|< 1$.  This implies that $C(u) \in V$ and $\log$ is defined at $C(u)$.  Thus,  we obtain $C(u) \in \rho(G)$ if and only if $\log(1+u)-\log(1-u) = \log(C(u)) \in  \log(V \cap \rho(G))  = U \cap d\rho (\mathfrak{g}) \subseteq d\rho(\mathfrak{g})$.
\end{proof}
By the proof of Lemma~\ref{crit:1},  the bound $1/3$ in Lemma \ref{crit:1} is independent of the choice of the norm on $\mathbb{V}$. To characterize representations of $G$ to which the Cayley transform is applicable,  we need the definition that follows.

\begin{definition}[Power span property]\label{cond:span}
Let $\mathfrak{u}$ be a Lie subalgebra of $\mathfrak{gl}(\mathbb{V})$.  An element $u \in \mathfrak{u}$ has the power span property if $u^{2k+1} \in \mathfrak{u}$ for any  integer $k \ge 0$.  A linear subspace $\mathbb{W} \subseteq \mathfrak{u}$ has the power span property if $x^{2k+1} \in \mathbb{W}$ for any element $x\in \mathbb{W}$ and any integer $k\ge 0$.  
\end{definition}
\begin{lemma}\label{cor: power span}
Suppose $u \in  d\rho (\mathfrak{g})$ and $\|u\| < \frac{1}{3}$.  If $u$ has the power span property,  then $C(u) \in \rho(G)$.
\end{lemma}
\begin{proof}
By definition,  $\log(1+u)-\log(1-u) = \sum_{k=0}^{\infty} c_{2k+1} u^{2k+1}$ for some real numbers $c_1,c_3,\dots$.  Moreover,  this series is convergent.  Since $u \in d\rho (\mathfrak{g})$ has the power span property,  we have $u^{2k+1} \in d\rho (\mathfrak{g})$ for each integer $k \ge 0$.  We notice that  $d\rho (\mathfrak{g})$ is finite dimensional,  thus $\log(1+u)-\log(1-u) \in d\rho (\mathfrak{g})$ and Lemma~\ref{crit:1} implies that $C(u) \in \rho(G)$.
\end{proof}

\begin{lemma}\label{cond:4}
Let $\mathfrak{u}$ be a Lie subalgebra of $\mathfrak{gl}(\mathbb{V})$.  Then the following are equivalent:
\begin{enumerate}[(a)]
\item $\mathfrak{u}$ has the power span property. \label{cond:4:item1}
\item If $a, b \in \mathfrak{u}$ then $aba\in \mathfrak{u}$.  \label{cond:4:item2}
\item If $a, b, c \in \mathfrak{u}$,  then $abc + cba\in \mathfrak{u}$.  \label{cond:4:item3}
\end{enumerate}
\end{lemma}
\begin{proof}
Implications \eqref{cond:4:item3} $\implies$ \eqref{cond:4:item2} $\implies$ \eqref{cond:4:item1} are obvious by induction.  Next we prove \eqref{cond:4:item1} $\implies$ \eqref{cond:4:item2}.  If $a,b \in \mathfrak{u}$ and $\mathfrak{u}$ has the power span property,  then $f(\mu) \coloneqq (a+\mu b)^3 \in \mathfrak{u}$ for each $\mu \in \mathbb{F}$.  In particular,  we have $f'(0) = a^2b +aba +ba^2 \in \mathfrak{u}$.  This implies that $aba \in \mathfrak{u}$ as $3aba = f'(0) - [a,[a,b]]$.  We complete the proof by deriving \eqref{cond:4:item3} from \eqref{cond:4:item2}.  For $a,b,c\in \mathfrak{u}$ and $\mu \in \mathbb{F}$,  we consider $g(\mu) \coloneqq (a+\mu c)b(a+\mu c)$.  By \eqref{cond:4:item2},  $g(\mu) \in \mathfrak{u}$ for any $\mu \in \mathbb{F}$.  This implies that $g'(0) = abc + cba \in \mathfrak{u}$. 
\end{proof}

In the following,  we will establish the equivalence between the applicability of the Cayley transform and the power span property.  To that end,  we need the next lemma.
\begin{lemma}\label{lem:Cpsp}
Let $G$ be a Lie group and let $\rho$ be a representation of $G$.  If $\rho: G \to \GL(\mathbb{V})$ is Cayley,  then $d\rho (\mathfrak{g})$ has the power span property.
\end{lemma}
\begin{proof}
By assumption,  there exists an open neighbourhood $V$ of $0\in \mathfrak{g}$ such that $C$ is defined on $d\rho (V)$ and $C(d\rho(V)) \subseteq \rho(G)$.  By shrinking $V$ if necessary,  we may assume that $\lVert d\rho (x) \rVert < 1/2$ for any $x \in V$.  For each $x\in \mathfrak{g}$,  there exists some $c > 0$ such that $c x \in V$.  It is clear that $d\rho (x)$ has the power span property if and only if $d\rho ( cx)$ has the power span property.  Hence it is sufficient to assume that $x \in V$.  We denote $u \coloneqq d\rho (x)$.  By definition,  $C(t u) \in \rho(G)$ for any $t \in \mathbb{F}$ with $\lvert \lambda \rvert \le 1$.  Lemma~\ref{crit:1} implies that $\varphi(t) \coloneqq \log(1+t u)-\log(1- t u) \in  d\rho(\mathfrak{g})$ if $t\in \mathbb{F}$ and $\lvert t \rvert \le 1$.  Therefore,  for any integer $k \ge 0$,  we may conclude that
\[
\varphi^{(2k+1)}(0) = 2 (2k)! u^{2k+1} \in d\rho(\mathfrak{g}).  \qedhere
\]
\end{proof}

\begin{theorem}[Characterization I]\label{cond:3}
Let $G$ be a Lie group and let $\rho$ be a representation of $G$.  The following are equivalent:
\begin{enumerate}[(a)]
\item $\rho$ is a Cayley representation. \label{cond:3:item1}
\item There exists an open neighbourhood $V$ of $0 \in \mathfrak{g}$ such that $\log(1 + d\rho(x)) - \log(1 - d \rho(x)) \in d\rho (\mathfrak{g})$ for any $x \in V$. \label{cond:3:item2}
\item $d\rho(\mathfrak{g})$ has the power span property.  \label{cond:3:item3}
\item If $a, b \in  d\rho(\mathfrak{g})$,  then $aba\in  d\rho(\mathfrak{g})$.  \label{cond:3:item4}
\item If $a, b, c \in d\rho(\mathfrak{g})$,  then $abc + cba\in  d\rho(\mathfrak{g})$. \label{cond:3:item5}
\end{enumerate}
\end{theorem}
\begin{proof}
The equivalence between \eqref{cond:3:item1} and \eqref{cond:3:item2} is a direct consequence of Lemma~\ref{crit:1},  while the equivalences among \eqref{cond:3:item3}--\eqref{cond:3:item5} follow from Lemma~\ref{cond:4}.  Moreover,  Lemma~\ref{lem:Cpsp} implies \eqref{cond:3:item1} $\implies$ \eqref{cond:3:item3}.  Lastly,  if \eqref{cond:3:item3} holds,  then Lemma~\ref{cor: power span} indicates that $C(u) \in \rho(G)$  for any $u\in d\rho(\mathfrak{g})$ with $\lVert u \rVert < 1/2$,  from which we may conclude that \eqref{cond:3:item1} holds.
\end{proof}

We recall that for an abelian Lie algebra $\mathfrak{g}$,  its universal enveloping algebra $\mathcal{U}(\mathfrak{g})$ is isomorphic to $S(\mathfrak{g}) \coloneqq \bigoplus_{n=0}^\infty S^n(\mathfrak{g})$,  where $S^n (\mathfrak{g}) \coloneqq \spn_\FF\left\{\prod_{k=1}^{n} u_i: u_i \in \mathfrak{g},  1\le  i \le n \right\}$ is the $n$-th symmetric power of $\mathfrak{g}$.  Thus,  we may simply identify the equivalence class $\overline{a \otimes b \otimes c} \in \mathcal{U}(\mathfrak{g})$ with $abc \in S^3(\mathfrak{g})$ for $a,b,c\in \mathfrak{g}$.  
\begin{proposition}[Characterization II]\label{prop:applicability2}
Let $G$ be an abelian Lie group and let $\mathfrak{g}$ be its Lie algebra.  Suppose $\rho: G \to \GL(\mathbb{V})$ is a representation of $G$.  Then the followings are equivalent:
\begin{enumerate}[(a)]
\item $\rho$ is a Cayley representation. \label{prop:applicability2:item1}
\item $\widetilde{d\rho} \left( S^3(\mathfrak{g}) \right) \subseteq d\rho(\mathfrak{g})$. \label{prop:applicability2:item2}
\item $\widetilde{d\rho}\left( S^{2k+1}(\mathfrak{g}) \right) \subseteq d\rho(\mathfrak{g})$ for any integer $k\ge 0$. \label{prop:applicability2:item3}
\end{enumerate}
Here $\widetilde{d\rho}:\mathcal{U} (\mathfrak{g})\simeq S(\mathfrak{g}) \to \mathfrak{gl}(\mathbb{V})$ is the algebra homomorphism induced by $d\rho$.
    \end{proposition}
    \begin{proof}
If $\widetilde{d\rho} \left( S^3(\mathfrak{g}) \right) \subseteq d\rho(\mathfrak{g})$,  then by induction we may derive $\widetilde{d\rho}\left( S^{2k+1}(\mathfrak{g}) \right) \subseteq d\rho(\mathfrak{g})$ for any integer $k\geq 0$.  This establishes the equivalence between \eqref{prop:applicability2:item2} and \eqref{prop:applicability2:item3}.  If \eqref{prop:applicability2:item2} holds and $x,y \in \mathfrak{g}$,  then $d\rho(x)d\rho(y)d\rho(x) = \widetilde{d\rho}(xyx)\in \widetilde{d\rho} \left( S^3(\mathfrak{g}) \right)\subseteq d\rho(\mathfrak{g})$.  Theorem \ref{cond:3} implies that $\rho$ is Cayley.  Conversely,  if \eqref{prop:applicability2:item1} holds and $x,y,z\in \mathfrak{g}$,  then by Theorem \ref{cond:3} we obtain
\[
\widetilde{d\rho} (xyz)=\frac{1}{2} \widetilde{d\rho} (xyz + zyx) =\frac{1}{2}\left( d\rho (x)  d\rho (y) d\rho (z)  + d\rho (z)  d\rho (y) d\rho (x) \right) \in d\rho (\mathfrak{g}). \qedhere
\]
\end{proof}

We conclude this section by a brief discussion on the applicability of the Cayley transform under an automorphism of $G$.  Let $\Aut(G)$ be the automorphism group of $G$.  Then for each $\sigma \in \Aut(G)$ and representation $\rho: G \to \GL(\mathbb{V})$,  we have another representation $\rho^{\sigma} \coloneqq \rho \circ \sigma: G \to \GL(\mathbb{V})$.  
\begin{proposition}\label{prop:auto}
The representation $\rho^{\sigma}$ has  the following properties:
\begin{enumerate}[(a)]
\item $\rho^{\sigma}$ is irreducible if and only if $\rho$ is irreducible.  \label{prop:auto:1}
\item If $\rho = \bigoplus_{i=1}^m \rho_i$ then $\rho^{\sigma} = \bigoplus_{i=1}^m \rho_i^{\sigma}$,  where $\rho_i: G \to \GL(\mathbb{V}_i)$ and $\bigoplus_{i=1}^m \mathbb{V}_i = \mathbb{V}$. \label{prop:auto:2}
\item If $\sigma \in \Inn(G)$,  then $\rho^{\sigma}$ and $\rho$ are isomorphic. \label{prop:auto:3}
\item $\rho^{\sigma}$ is Cayley if and only if $\rho$ is Cayley.  \label{prop:auto:4}
\end{enumerate}
\end{proposition}
\begin{proof}
All these properties can be easily verified by definition.
\end{proof}
According to Proposition~\ref{prop:auto}~\eqref{prop:auto:3},  the outer automorphism group $\Out(G) \coloneqq \Aut(G)/\Inn(G)$ acts on the representation ring $R(G)$ generated by isomorphism classes of representations of $G$.  Moreover,  by Proposition~\ref{prop:auto}~\eqref{prop:auto:4},  the applicability of the Cayley transform is stable with respect to the action of $\Out(G)$.
\begin{corollary}[Stability I]\label{cor:stability-1}
If $\rho$ is a Cayley representation,  then all representations in the $\Out(G)$-orbit of $\rho$ in $R(G)$ are Cayley. 
\end{corollary}
\begin{remark}\label{rmk:spin8}
According to \cite[Section~20.3]{fulton2013representation},  $\Out(\Spin_8(\CC))$ acts on $\{\mathbb{C}^n,   S^+, S^-\}$ as the permutation group $\mathfrak{S}_3$.  Here $\mathbb{S}_+, \mathbb{S}_-$ are the two spin representations.  Since $\mathbb{C}^n$ is Cayley,  Corollary~\ref{cor:stability-1} immediately implies that $\mathbb{S}_+$ and $\mathbb{S}_-$ are also Cayley representations. 
\end{remark}
\section{The Cayley transform for semisimple Lie groups}\label{sec:semisimple}
%
%
%
%
Let $G$ be a semisimple Lie group over $\FF$ and let $\mathfrak{g}$ be its Lie algebra.  Assume $\rho: G \to \GL(\mathbb{V})$ is a representation of $G$ and $\mathfrak{h}$ is a Cartan subalgebra of $\mathfrak{g}$.  We recall that if $\mathbb{F} = \mathbb{C}$,   then 
\[
\mathfrak{g} = \mathfrak{g}_0 \bigoplus \left( \bigoplus_{\alpha\in \Phi} \mathfrak{g}_{\alpha} \right),\quad \mathfrak{g}_0 = \mathfrak{h},
\]
where $\Phi$ is the set of roots of $\mathfrak{g}$.  If $\mathfrak{g}$ is a compact real form,  then we have 
\[
\mathfrak{g} =\mathfrak{g}_0  \bigoplus \left( \bigoplus_{\alpha\in \Phi} \left( \mathfrak{g}^{\mathbb{C}}_{\alpha} \bigoplus \mathfrak{g}^{\mathbb{C}}_{-\alpha} \right) \bigcap \mathfrak{g} \right),\quad \mathfrak{g}_0 = \mathfrak{h}.
\]
Moreover,  if $\mathfrak{g}$ is a non-compact real form,  then $\mathfrak{g}$ admits a restricted root space decomposition \cite[Proposition~6.40 and Equation~(6.48b)]{Knapp02}:
\[
\mathfrak{g} = \mathfrak{g}_0 \bigoplus \left( \bigoplus_{\sigma \in \Sigma} \mathfrak{g}_{\sigma} \right),\quad \mathfrak{g}_{\mu} = \mathfrak{g} \bigcap \left(  \bigoplus_{ \alpha \in \Phi,  \; \alpha|_{\mathfrak{a}}  = \mu } \mathfrak{g}^{\mathbb{C}}_{\alpha}\right),\quad  \mu \in \Sigma \sqcup \{0\},
\]
where $\mathfrak{a}$ is some abelian subspace of $\mathfrak{g}$ and $\Sigma$ is the set of restricted roots of $\mathfrak{g}$.  In this case,  $\mathfrak{g}_0 = \mathfrak{h}$ is not necessarily true,  but it is true if $\mathfrak{g}$ is a split real form.
\begin{lemma}\label{lem: a3 to a}
Suppose that $G$ is a semisimple Lie group over $\mathbb{F}$ such that $\mathfrak{g}_0 = \mathfrak{h}$.  If $\rho$ is Cayley,  then $\widetilde{d\rho} \left( S^3\mathfrak{h}\right) \subseteq d\rho(\mathfrak{h})$. 
\end{lemma}
\begin{proof}
By Propositions~\ref{prop: group to Lie algebra} and \ref{prop:applicability2},  it is clear that $\widetilde{d\rho}  \left( S^3 \mathfrak{h} \right) \subseteq d\rho(\mathfrak{g})$.  Next we split our discussion into two cases: 
\begin{itemize}
\item[$\diamond$] $\FF=\CC$.  Let $\mathbb{V} = \bigoplus_{\lambda \in \Lambda} \mathbb{V}_{\lambda}$ be the weight space decomposition.  Here $\Phi$ (resp.  $\Lambda$) is the set of roots (resp.  weight lattice) of $\mathfrak{g}$.  By definition,  we have 
\[
d\rho (x) (\mathbb{V}_{\lambda}) \subseteq \begin{cases}
\mathbb{V}_{\lambda + \alpha} &\text{~if~}x \in \mathfrak{g}_{\alpha} \\
\mathbb{V}_{\lambda} &\text{~if~}x \in \mathfrak{h} \\
\end{cases}.
\]
Since $d\rho (\mathfrak{g}) = d\rho (\mathfrak{h}) \bigoplus \left( \oplus_{\alpha \in \Phi} d\rho (g_{\alpha}) \right)$ and $\widetilde{d\rho}  \left( S^3 \mathfrak{h} \right) \subseteq d\rho(\mathfrak{g})$,  we conclude that $\widetilde{d\rho}  \left( S^3 \mathfrak{h} \right) \subseteq d\rho(\mathfrak{h})$.
\item[$\diamond$] $\mathbb{F} = \mathbb{R}$.  We consider $\Ind(d\rho): \mathfrak{g}^{\mathbb{C}} \to \mathfrak{gl}(\mathbb{V})$.  Since $\mathfrak{h}^{\mathbb{C}}$ is a Cartan subalgebra of $\mathfrak{g}^{\mathbb{C}}$,  by the argument for $\mathbb{F} = \mathbb{C}$ we have 
\[
\widetilde{\Ind(d\rho)} \left( S^3  \mathfrak{h}^{\mathbb{C}} \right) \subseteq \Ind(d\rho) \left( \mathfrak{h}^{\mathbb{C}} \right).
\]
Therefore,  we obtain 
\[
\widetilde{d\rho} \left( S^3 \mathfrak{h} \right) \subseteq  \widetilde{\Ind(d\rho)} \left( S^3  \mathfrak{h}^{\mathbb{C}}  \right) \bigcap d\rho (\mathfrak{g}) \subseteq \Ind(d\rho) \left( \mathfrak{h}^{\mathbb{C}} \right) \bigcap d\rho (\mathfrak{g}).
\]
By assumption,  we have a decomposition $\mathfrak{g} = \mathfrak{h} \bigoplus \mathfrak{l}$ for some subspace $\mathfrak{l} \subseteq \mathfrak{g}$.  Thus,  $d\rho (\mathfrak{g}) = d \rho (\mathfrak{h}) \oplus d \rho (\mathfrak{l})$ and $\Ind(d\rho) \left( \mathfrak{h}^{\mathbb{C}} \right) \bigcap d\rho (\mathfrak{g}) = d\rho (\mathfrak{h})$.  \qedhere
\end{itemize}
\end{proof}
As a consequence of Lemma~\ref{lem: a3 to a},  we obtain the third characterization of the applicability of the Cayley transform.  
\begin{theorem}[Characterization III]\label{cond:simple}
Let $G$ be a semisimple Lie group over $\mathbb{F}$ such that $\mathfrak{g}_0 = \mathfrak{h}$ and $\mathfrak{g} = \Ad(G) (\mathfrak{h})$ where $\Ad: G \to \GL(\mathfrak{g})$ is the adjoint representation of $G$.  Let $\rho: G \to \GL(\mathbb{V})$ be a representation of $G$.  Then $\rho$ is Cayley if and only if $ \widetilde{d \rho}\left( S^3 \mathfrak{h} \right) \subseteq d \rho (\mathfrak{h})$. 
\end{theorem}
\begin{proof}
By Theorem~\ref{cond:3} and Lemma~\ref{lem: a3 to a},  it suffices to prove that $\widetilde{d \rho}\left( S^3 \mathfrak{h} \right) \subseteq d \rho (\mathfrak{h})$ ensures the power span property of $\mathfrak{g}$.  To this end,  we notice that each $x \in \mathfrak{g}$ can be written as $x = \Ad(g) H$ for some $g\in G$ and $H \in \mathfrak{h}$.  Hence,  we have $d\rho(x) = d \rho ( \Ad(g) H ) = \rho(g) \left( d \rho(H) \right) \rho(g)^{-1}$,  which together with Proposition~\ref{prop:applicability2} leads to
\[
d\rho(x)^{2k+1} = \rho(g) \left( d \rho(H) \right)^{2k+1} \rho(g)^{-1} \in \rho(g) d \rho(\mathfrak{h}) \rho(g)^{-1} = d\rho (\Ad(g)(\mathfrak{h})) \subseteq d\rho (\mathfrak{g}).  \qedhere
\]
\end{proof}
\subsection{The Cayley transform on irreducible representations} 
In this subsection,  we  consider the applicability of the Cayley transform on an irreducible representation of a semisimple Lie group.  Let $\mathfrak{g}$ be a semisimple Lie algebra over $\FF$ and let $\pi: \mathfrak{g} \to \mathfrak{gl}(\mathbb{V})$ be a representation.  We suppose that $\mathbb{V}=\bigoplus_{\lambda \in \Lambda}\mathbb{V}_\lambda$ is the weight space decomposition of $\mathbb{V}$,  where $\Lambda$ is the weight lattice of $\mathfrak{g}$.  Moreover,  we denote the set of simple roots of $\mathfrak{g}$ by $\Delta^+$.  
\subsubsection{Geometric condition} Since a finite dimensional irreducible  representation of a semisimple Lie group are characterized by its highest weight,  the applicability of the Cayley transform on this representation must impose a condition on the highest weight.  To formulate this condition explicitly,  we first establish the following three lemmas. 
\begin{lemma}\label{prop: balance weights}
We have $\sum_{\lambda \in \Lambda} \dim ( \mathbb{V}_\lambda)\lambda =0$.
\end{lemma}
\begin{proof}
For simplicity,  we denote $\omega  \coloneqq  \sum_{\lambda \in \Lambda} \dim ( \mathbb{V}_\lambda)\lambda$.  Assume $\FF=\CC$.  We proceed by contradiction.  Suppose that $\omega  \ne 0$.  There exists some $H\in \mathfrak{h}$ such that $\omega  (H) \ne 0$.  Since $\mathfrak{h} = \spn_{\CC} \{ H_{\alpha}: \alpha \in \Delta^+ \}$,  we can find some $\beta \in \Delta^+$ such that $\omega (H_{\beta}) \ne 0$.  Here $H_{\alpha} \in [\mathfrak{g}_{\alpha},  \mathfrak{g}_{-\alpha}]$ is the co-root of $\alpha \in \Delta^+$.  By definition,  $H_{\beta} = [x_{\beta},  x_{-\beta}]$ for some $x_{\beta} \in \mathfrak{g}_{\beta}$ and $x_{-\beta} \in \mathfrak{g}_{-\beta}$.  Thus,  we obtain a contradiction: 
\[
\omega  (H_{\beta}) = \tr (\pi (H_{\beta})) = \tr\left( [\pi(x_\beta),  \pi(x_{-\beta})]  \right)= 0.
\]
For $\FF=\RR$,  we observe that $\omega$ remains unchanged if we pass from $\pi$ to  $\Ind \pi: \mathfrak{g}^{\mathbb{C}} \to \mathfrak{gl}(\mathbb{V})$.  This implies that $\omega = 0$ and the proof is complete.
\end{proof}
For a subset $S$ of a free Noetherian $\mathbb{Z}$-module $M$,  we denote by $\mathbb{Z} S$ the submodule of $M$ generated by $S$.  The \emph{rank of $S$} is defined by $\rank(S) \coloneqq \rank (\mathbb{Z} S)$.  In particular,  $\rank (S)$ is equal to the $\mathbb{K}$-dimension of the subspace of $M\otimes_{\ZZ} \mathbb{K}$ spanned by $S$ over any field $\mathbb{K}$ of characteristic zero.
\begin{lemma}\label{lem:rank-orbit}
If $\pi: \mathfrak{g} \to \mathfrak{gl}(\mathbb{V})$ is an irreducible representation with the highest weight $\mu$ and $\dim_\FF (\pi(\mathfrak{h}))= \dim_\FF(\mathfrak{h})$,  then $\rank ( \mathcal{O}_\mu ) = \dim_\FF(\mathfrak{h})$ where $\mathcal{O}_\mu \coloneqq \{s(\mu): s\in \W \}$ and $\W$ is the Weyl group of $\mathfrak{g}$.
\end{lemma}
\begin{proof}
Let $n \coloneqq \dim_{\FF} (\mathfrak{h})$. Suppose that elements in $S \coloneqq \{\lambda\in \Lambda: \mathbb{V}_{\lambda} \ne 0\}$ are $\lambda_1,  \dots,  \lambda_m$.  We define
\begin{equation}\label{lem:rank-orbit:eq1}
A = [
\underbrace{\lambda_1,\dots, \lambda_1}_{\dim (\mathbb{V}_{\lambda_1}) \text{~copies}},\dots,  \underbrace{\lambda_m,\dots, \lambda_m}_{\dim (\mathbb{V}_{\lambda_m}) \text{~copies}} ] \in \Lambda^{\dim (\mathbb{V})}.
\end{equation}
If we regard each $\lambda \in \Lambda \subseteq \mathfrak{h}^\ast$ as an $n$-dimensional column vector,  then $A  \in \mathbb{F}^{n \times \dim (\mathbb{V})}$.  By Theorem~\ref{thm: FTRTSLA}, each element in $S$ is a convex linear combination of elements in $\mathcal{O}_{\mu}$.  Therefore,  we have $\rank( A ) = \rank( S ) = \rank(\mathcal{O}_{\mu})$.  Next we observe that $\pi(\mathfrak{h}) = \left\lbrace
A(H): H\in \mathfrak{h}
\right\rbrace$,  where 
\begin{equation}\label{lem:rank-orbit:eq2}
A(H) \coloneqq \diag \left(
\underbrace{\lambda_1(H),\dots, \lambda_1(H)}_{\dim (\mathbb{V}_{\lambda_1}) \text{~copies}},\dots,  \underbrace{\lambda_m(H),\dots, \lambda_m(H)}_{\dim (\mathbb{V}_{\lambda_m}) \text{~copies}} \right) \in \mathbb{C}^{\dim (\mathbb{V}) \times \dim (\mathbb{V})}.
\end{equation}
This implies that $\dim_{\mathbb{F}} (\mathfrak{h}) = \dim_{\mathbb{F}} (\pi (\mathfrak{h})) \le \rank(A) = \rank(\mathcal{O}_{\mu}) \le \dim_{\mathbb{F}} (\mathfrak{h})$.
\end{proof}

\begin{lemma}\label{prop: Gauss elimination}
If $\pi: \mathfrak{g} \to \mathfrak{gl}(\mathbb{V})$ is an irreducible representation such that $\widetilde{\pi} \left( S^3 \mathfrak{h}\right)\subseteq \pi(\mathfrak{h})$ and  $\dim_\FF( \mathfrak{h} ) = \dim_\FF(  \pi(\mathfrak{h}) ) \eqqcolon n$,  then there exist {linearly independent} $\omega_1,\dots,  \omega_n \in \Lambda$ such that 
\[
\{\omega_1,\dots,  \omega_n\} \subseteq 
\{\lambda\in \Lambda\setminus \{0\}: \mathbb{V}_{\lambda} \ne 0\}
\subseteq 
\{\pm \omega_1,\dots,  \pm \omega_n\}. 
\]
\end{lemma}
\begin{proof}
Suppose $\Delta^+ = \{\alpha_1,\dots,  \alpha_n\}$.  Let $H_j$ be the co-root of $\alpha_j \in \Delta^+$ for each $1 \le j \le n $.  Since $\{H_1,\dots,  H_n\}$ is a basis of $\mathfrak{h}$,  it has the dual basis $\{\omega'_1,\dots,  \omega'_n\}$.  Assume $\{\lambda\in \Lambda: \mathbb{V}_{\lambda} \ne 0\}= \{ \lambda_1,  \dots,  \lambda_m \}$.  By definition,  there exists column vectors $a_1,\dots,  a_m \in \mathbb{Z}^{n}$ such that 
\begin{equation}\label{prop: Gauss elimination:eq1}
A =  \omega'  A',  \quad \omega'\coloneqq  \begin{bmatrix}
\omega'_1 &
\cdots &
\omega'_n
\end{bmatrix},\quad A' \coloneqq [
\underbrace{a_1,  \cdots,  a_1}_{\dim(\mathbb{V}_1) \text{~copies}},  \cdots,  \underbrace{a_m,  \cdots,  a_m}_{\dim(\mathbb{V}_m) \text{~copies}} ]
\end{equation}
where $A$ is defined in \eqref{lem:rank-orbit:eq1}.  Clearly,  $\omega'(\mathfrak{h}) \subseteq \mathbb{C}^n$ consists of vectors $\omega'(H) \coloneqq \begin{bsmallmatrix} \omega'_1(H),\cdots,  \omega'_n(H)
\end{bsmallmatrix}$ for $H\in \mathfrak{h}$.  In fact,  we must have $\omega'(\mathfrak{h}) = \mathbb{C}^n$ as $\rank (\omega') = n$.  Let $A_0 \coloneqq [
a_1, \cdots,  a_m] \in \mathbb{Z}^{n \times m}$.
\begin{itemize}
\item[$\diamond$] $\FF=\CC$.  According to the proof of Lemma~\ref{lem:rank-orbit},  we have $\rank(A_0) = n$.  We may re-index $\lambda_j$'s such that the left $n\times n$ submatrix of $A_0$ has the full rank.  By Gaussian elimination,  we obtain a decomposition $A_0 = GB$ where $G \in \mathbb{Z}^{n\times n}$ and 
\[
B = \begin{bmatrix}
I_n & C 
\end{bmatrix} = \begin{bmatrix} 
1 & 0 & \cdots & 0 & c_{1,1} & \cdots & c_{1,m-n} \\
0 & 1 & \cdots & 0 & c_{2,1} & \cdots & c_{2,m-n} \\
\vdots & \vdots & \ddots & \vdots & \vdots & \ddots & \vdots \\
0 & 0 & \cdots & 1 & c_{n,1} & \cdots & c_{n,m-n}
\end{bmatrix} \in \mathbb{Q}^{n \times m}.
\]
It is clear that we can construct $B' \in \mathbb{Q}^{n \times \dim (\mathbb{V})}$ from $B$ such that $A' = G B'$.  We take $\omega \coloneqq \omega' G \in \Lambda^n$.  By \eqref{prop: Gauss elimination:eq1} we have $A(H) = \diag ( \omega'(H) A' ) =  \diag ( \omega(H) B' )$ for any $H\in \mathfrak{h}$.  Since $S^3 \left( \pi(\mathfrak{h}) \right) = \widetilde{\pi} \left( S^3 \mathfrak{h}\right)\subseteq \pi(\mathfrak{h})$,  \eqref{lem:rank-orbit:eq2} implies that $A(H)^2 A(H') \in \pi (\mathfrak{h})$ for any $H,H'\in \mathfrak{h}$.  Let $X_1,\dots,  X_n$ be the dual basis of $\omega_1,\dots  \omega_n$.  Then by evaluating $A(X_j)^2 A(X_k)$ we obtain 
\[
c_{js}^3 = c_{js},\quad c_{js}^2 c_{ks} = 0,\quad 1 \le s \le m-n,\; 1 \le j < k \le n.
\]
Therefore,  $c_{js} \in \{0,-1,1\}$ for all $1 \le s \le m-n$ and $1 \le j  \le n$.  Moreover,  for each $1 \le s \le m-n$ there is at most one $1 \le j \le n$ such that $c_{js} \ne 0$.  By \eqref{prop: Gauss elimination:eq1},  we derive that $A = \omega B'$ where $B' \in \{0,\pm 1\}^{n \times \dim (\mathbb{V})}$.  By construction,  we clearly have $\omega = (\lambda_1,\dots,  \lambda_n)^\tp$ and this completes the proof.
\item[$\diamond$] $\FF=\RR$.  We notice that $        \widetilde{\pi} \left( S^3 \mathfrak{h} \right) \subseteq  \pi \left( \mathfrak{h} \right) $ implies 
\[
\widetilde{\Ind(\pi)} \left( S^3 (\mathfrak{h}^{\CC}) \right) = \widetilde{\pi} \left( S^3 \mathfrak{h} \right)\otimes \mathbb{C} \subseteq \pi \left( \mathfrak{h} \right) \otimes \CC = \Ind( \pi)\left( \mathfrak{h}^{\CC}\right).
\]
Since weights of $\pi$ are the same as those of $\Ind(\pi)$,  we obtain the desired inclusions by the complex case.  \qedhere
\end{itemize}
\end{proof}

\begin{theorem}[Geometric condition]\label{thm:main}
Suppose that $G$ is a semisimple Lie group over $\mathbb{F}$ such that $\mathfrak{g}_0 = \mathfrak{h}$.  Let $\rho: G \to \GL(\mathbb{V})$ be an  irreducible representation with the highest weight $\omega_1$ such that $\dim_\FF(\mathfrak{h}) = \dim_\FF(d \rho (\mathfrak{h})) \eqqcolon n$.  If $\rho$ is Cayley,  then we have the following: 
\begin{enumerate}[(a)]
\item Either $\{\lambda \in \Lambda:  V_{\lambda} \ne 0 \}= \mathcal{O}_{\omega_1}$ or  $\{\lambda \in \Lambda:  V_{\lambda} \ne 0 \} =  \mathcal{O}_{\omega_1} \sqcup \{0\}$.  In particular,  a weight of $\mathbb{V}$ is either zero or an extreme weight.
\label{thm:main:item1}
\item $\rank ( \mathcal{O}_{\omega_1} ) = n$.  \label{thm:main:item2}
\item There exist $\omega_2,\dots,  \omega_n \in \Lambda$ such that $\mathcal{O}_{\omega_1} = \{ \pm \omega_1,\dots,  \pm \omega_n \}$.  \label{thm:main:item3}
\end{enumerate}
\end{theorem}
\begin{proof}
By Lemma~\ref{lem:rank-orbit},  it is clear that \eqref{thm:main:item2} holds.  We prove \eqref{thm:main:item1} and \eqref{thm:main:item3} at the same time.  According to Lemmas~\ref{lem: a3 to a} and \ref{prop: Gauss elimination},  there exist $\omega'_1, \omega_2,  \dots,  \omega_n \in \Lambda$ such that 
\begin{equation}\label{thm:main:eq1}
\{\omega'_1,\dots,  \omega_n\} \subseteq 
\{\lambda\in \Lambda\setminus \{0\}: \mathbb{V}_{\lambda} \ne 0\}
\subseteq 
\{\pm \omega'_1,\dots,  \pm \omega_n\}. 
\end{equation}
Without loss of generality,  we may assume that $\omega'_1$ or $-\omega'_1$ is the highest weight of $\mathbb{V}$.  By swapping $\Delta^+$ and $\Delta^-$ if necessary,  we may further assume that $\omega'_1 = \omega_1$.  For simplicity,  we denote 
\[
\mathcal{O} \coloneqq \mathcal{O}_{\omega_1},\quad S \coloneqq \{\lambda\in \Lambda: \mathbb{V}_{\lambda} \ne 0\}.
\] 

We first claim that $\omega_j \in \mathcal{O}$ for all $1 \le j \le n$.  Indeed,  if on the contrary that $\omega_j \not\in \mathcal{O}$ for some $j$,  then we may assume that $\omega_2 \not\in \mathcal{O}$.  We write $\mathcal{O} = \{ \mu_1,\dots,  \mu_m \}$ where $\mu_1 = \omega_1$.  We recall that $\mathcal{O} \subseteq S \setminus \{0\}$ and $\omega_2 \in S = \Conv( \mathcal{O} ) \cap \left( \omega_1 + \mathbb{Z} \Phi \right)$ where $\Conv( \mathcal{O} )$ is the convex hull of $\mathcal{O}$ and $ \mathbb{Z} \Phi$ is the root lattice of $\mathfrak{g}$.  Hence,  we have 
\[
\omega_2 = \sum_{i=1}^{m} c_i \mu_i,  \quad \sum_{i=1}^{m} c_i = 1, \quad   0 \leq c_i < 1, \; 1 \le i \le m.
\]
If $- \omega_2 \notin \mathcal{O}$ then $\mathcal{O} \subseteq \{\pm \omega_1, \pm \omega_3,  \dots,  \pm \omega_n\}$.  This implies that $\omega_2 \in S \subseteq  \spn\{ \omega_1, \omega_3,\dots \omega_n\}$,  which contradicts to the linear independence of $\omega_1,\dots,  \omega_n$.  If $- \omega_2 \in \mathcal{O}$,  then $-\omega_2 = \mu_j$ for some $1 \le j  \le n$.  This leads to a contradiction again:
\[
(1+c_j) \omega_2 = \sum_{i \ne j} c_i \mu_i \in  \spn\{ \mu_i:  1\le i \ne j \le m\} \subseteq \spn\{\pm \omega_1, \pm \omega_3,  \dots,  \pm \omega_n\}.
\]

Next,  we prove that there is some $1 \le k \le n$,  such that $-\omega_k \in \mathcal{O}$.  Otherwise,  by \eqref{thm:main:eq1} we must have $\mathcal{O} = \{\omega_1,\dots,  \omega_n\}$.  Then $\mathcal{O}  \subseteq S  \subseteq \Conv( \mathcal{O} )$ and \eqref{thm:main:eq1} imply $S \setminus \{0\} = \mathcal{O}$.  The linear independence of $\omega_1,\dots,  \omega_n$ indicates that $\sum_{j=1}^n \omega_n \ne 0$.  Since $\dim \mathbb{V}_{\omega_j} = 1$ for each $1 \le j \le n$,  this contradicts to Lemma~\ref{prop: balance weights}.

Lastly,  $\pm \omega_k \in \mathcal{O}$ if and only if $\pm \omega_j \in \mathcal{O}$ for each $1 \le j \le n$,  since $ \mathcal{O}$ is an orbit of $\W$.  Thus,  \eqref{thm:main:eq1} implies that $\mathcal{O} = \{\pm \omega_1,\dots,  \pm \omega_n\} = S \setminus \{0\}$.  This completes the proof of \eqref{thm:main:item1} and \eqref{thm:main:item3}.
\end{proof}
Theorem~\ref{thm:main} characterizes the set of weights of an irreducible representation $\mathbb{V}$,  on which the Cayley transform is applicable.  As a direct consequence,  one can also describe its weight diagram.  We recall that the weight diagram of a representation $\mathbb{V}$ is a pair $(V,  m)$,  where $V \coloneqq \{\lambda \in \Lambda: \mathbb{V}_{\lambda} \ne 0\}$ and $m: V \to \mathbb{N}$ is defined by $m(\lambda) = \dim V_{\lambda}$.
\begin{corollary}[Weight diagram]\label{rmk:polytope}
{Let $(V,m)$ be the weight diagram of $\mathbb{V}$.  Then $V$ is a $\W$-invariant subset $\Lambda \simeq \mathbb{Z}^n$,  which is not contained in any low rank sublattice.  It is symmetric about the origin and has $2n$ or $(2n+1)$ vertices,  $2n$ of which form a $\W$-orbit and $m = 1$ on this orbit.  In the latter case,  the origin is the additional vertex. }
\end{corollary}
\begin{definition}[Cayley configuration]
An irreducible representation of a semisimiple Lie group is called a representation with Cayley configuration if its weight diagram has the configuration described in Corollary~\ref{rmk:polytope}.
\end{definition}

\begin{corollary}[Self-duality]
For each representation $\rho: G \to \GL(\mathbb{V})$ with Cayley configuration,  we have $\mathbb{V} \simeq \mathbb{V}^\ast$ as representations of $G$.  In particular,  If $\mathfrak{g}_0 = \mathfrak{h}$ and $\rho: G \to \GL(\mathbb{V})$ is an  irreducible representation such that $\dim_\FF(\mathfrak{h}) = \dim_\FF(d \rho (\mathfrak{h}))$ and  $\rho$ is Cayley,  then $\mathbb{V}$ is self-dual.
\end{corollary}
\begin{proof}
By definition and the symmetry of the weight diagram of a representation with Cayley configuration,  we have 
\[
\{\lambda\in \Lambda: \mathbb{V}_{\lambda} \ne 0 \} = -\{\lambda\in \Lambda: (\mathbb{V}^\ast)_{\lambda} \ne 0 \} = \{\lambda\in \Lambda: (\mathbb{V}^\ast)_{\lambda} \ne 0 \},
\]
from which we conclude that $\mathbb{V} \simeq \mathbb{V}^\ast$.
\end{proof}

We recall from Section~\ref{sec:Cayley} that $\Out(G)$ acts on $R(G)$.  By Theorem~\ref{thm: FTRTSLA},  this induces an action of $\Out(G)$ on the weight lattice $\Lambda$.  In particular,  $\Out(G)$ also acts on the subset $R_0(G)$ consisting of isomorphism classes of irreducible representations.  The following observation is obvious .
\begin{corollary}[Stability II]
If $\mathbb{V} \in R_0(G)$ is a representation with Cayley configuration,  then for each $\varphi \in \Out(G)$,  $\varphi (\mathbb{V})$ is also a representation with Cayley configuration.
\end{corollary}
\subsubsection{The sufficiency of the geometric condition}Theorem~\ref{thm:main} (or equivalently,   Corollary~\ref{rmk:polytope}) is a necessary condition for the applicability of the Cayley transform on an irreducible representation.  Next we investigate its sufficiency.  We begin with the existence of a representation with Cayley configuration.
\begin{lemma}[Existence of representation with Cayley configuration]\label{prop:orbit to complex contain}
Let $G$ be a simply connected semisimple Lie group over $\mathbb{F}$ and let $\mathcal{O}$ be the $\W$-orbit of $\omega_1 \in \Lambda$.  Suppose that $\mathcal{O}$ is symmetric about the origin,  $\dim_{\mathbb{F}} (\mathfrak{h}) = \rank(\mathcal{O}) = |\mathcal{O}|/2 = n$ and $\Conv(\mathcal{O}) \cap \left( \omega_1 + \mathbb{Z} \Phi \right) \subseteq \mathcal{O} \cup \{ 0 \}$.   Then there exists a unique irreducible representation $\rho: G \to \GL(\mathbb{V})$ such that 
\begin{enumerate}[a)]
\item $\{\lambda \in \Lambda: \mathbb{V}_{\lambda} \ne 0 \} = \Conv(\mathcal{O}) \cap (\omega_0 + \mathbb{Z} \Phi)$.  \label{prop:orbit to complex contain:item1}
\item $\dim \left( d \rho (\mathfrak{h} ) \right) = n$. \label{prop:orbit to complex contain:item2}
\item $\widetilde{ d \rho } \left( S^3 \mathfrak{h} \right) \subseteq d \rho (\mathfrak{h})$ for $\mathbb{F} = \mathbb{C}$ and $\widetilde{ d \Ind (\rho) } \left( S^3 \mathfrak{h}^{\mathbb{C}} \right) \subseteq d \Ind(\rho) (\mathfrak{h}^{\mathbb{C}})$ for $\mathbb{F} = \mathbb{R}$.  \label{prop:orbit to complex contain:item3}
\end{enumerate}
\end{lemma}
\begin{proof}
If such an irreducible representation exists,  then its uniqueness follows immediately.  Since $\mathcal{O}$ is symmetric about the origin and $|\mathcal{O}| = 2n$,  we may write $\mathcal{O} = \{ \pm \omega_1,\dots,  \pm \omega_n \}$.  The assumption that $\rank (\mathcal{O}) = n$,  $\omega_1,\dots,  \omega_n$ are linearly independent.  Let $C$ be the closed Weyl chamber of $\Lambda$ containing $\omega_1$.  By definition,  $\Lambda = \W C$ and $\omega_1 \in \mathcal{O} \cap C$.  Let $\rho: G \to \GL(\mathbb{V})$ be the irreducible representation associated to $\omega_1$.  We have 
\[
\mathcal{O} \subseteq \{\lambda \in \Lambda: \mathbb{V}_{\lambda} \ne 0 \} = \Conv(\mathcal{O}) \cap (\omega_0 + \mathbb{Z} \Phi) \subseteq \mathcal{O} \cup \{ 0 \}.
\]
\begin{itemize}
\item[$\diamond$] $\mathbb{F} = \mathbb{C}$.  Let $H_1,\dots,  H_n \in \mathfrak{h}$ be the dual basis of $\omega_1,\dots,  \omega_n$.  It is obvious that $\omega_j (H_k)^3 = \omega_j (H_k)$ for any $1 \le j,  k \le n$.  For each $H \in \mathfrak{h}$,  we write $H = \sum_{j=1}^n c_j H_j$.  Therefore,  by \eqref{lem:rank-orbit:eq2} we obtain that $\dim \left( d \rho (\mathfrak{h}) \right) = n$ and 
\[
\left( d \rho (H) \right)^3 = 
\begin{cases}
\diag \left( 
c^3_1, - c^3_1,  \dots,  c^3_n, - c^3_n \right) &\text{~if~} \Conv(\mathcal{O}) \cap (\omega_0 + \mathbb{Z} \Phi) = \mathcal{O} \\
\diag \left(c^3_1, - c^3_1,  \dots,  c^3_n, - c^3_n, 0  \right) &\text{~otherwise}
\end{cases}.
\]
This implies that $\widetilde{d \rho} \left(  S^3\mathfrak{h}\right) \subseteq d\rho (\mathfrak{h})$.
\item[$\diamond$] $\mathbb{F} = \mathbb{R}$.  Let $\rho^{\mathbb{C}}: G^{\mathbb{C}} \to \GL(\mathbb{V})$ be the representation satisfying \eqref{prop:orbit to complex contain:item1}--\eqref{prop:orbit to complex contain:item3} for $G^{\mathbb{C}}$.  By construction,  we have $\{\lambda \in \Lambda: \mathbb{V}_{\lambda} \ne 0 \} = \Conv(\mathcal{O}) \cap (\omega_0 + \mathbb{Z} \Phi) $.  It is clear that $\rho = \Res(\rho^{\mathbb{C}}): G \to \GL(\mathbb{V})$ is the desired representation since $\Ind \left( \Res \left( \rho^{\mathbb{C}}\right) \right) = \rho^{\mathbb{C}}$.  \qedhere
\end{itemize}
\end{proof}

\begin{proposition}\label{prop: split compact back to real}
Let $G$ be a simply connected compact or split real form of the complex semisimple Lie group $G^{\mathbb{C}}$ and let $\mathcal{O}$ be the $\W$-orbit of $\omega_1 \in \Lambda$.  Suppose that $\mathcal{O}$ is symmetric about the origin,  $\dim (\mathfrak{h}) = \rank(\mathcal{O}) = |\mathcal{O}|/2 = n$ and $\Conv(\mathcal{O}) \cap \left( \omega_1 + \mathbb{Z} \Phi \right) \subseteq \mathcal{O} \cup \{ 0 \}$.   Then there exists a unique irreducible representation $\rho: G \to \GL(\mathbb{V})$ such that 
\begin{enumerate}[a)]
\item $\{\lambda \in \Lambda: \mathbb{V}_{\lambda} \ne 0 \} = \Conv(\mathcal{O}) \cap (\omega_0 + \mathbb{Z} \Phi)$.  \label{prop: split compact back to real:item1}
\item $\dim \left( d \rho (\mathfrak{h} ) \right) = n$. \label{prop: split compact back to real:item2}
\item $\widetilde{ d \rho } \left( S^3 \mathfrak{h} \right) \subseteq d \rho (\mathfrak{h})$.  \label{prop: split compact back to real:item3}
\end{enumerate}
\end{proposition}
\begin{proof}
Let $\rho: G \to \GL(\mathbb{V})$ be the irreducible representation of $G$ in Lemma~\ref{prop:orbit to complex contain}.  Then we have \eqref{prop: split compact back to real:item1},  \eqref{prop: split compact back to real:item2} and $\widetilde{d\rho} \left( S^3 \mathfrak{h} \right) \otimes \mathbb{C} \subseteq d\rho (\mathfrak{h}) \otimes \mathbb{C}$.  
\begin{itemize}
\item[$\diamond$] If $G$ is the split form,  then $\lambda(H) \in \mathbb{R}$ for any $\lambda \in \Lambda$ and $H\in \mathfrak{h}^{\mathbb{C}}$.  By \eqref{lem:rank-orbit:eq2},  $\widetilde{d\rho} \left( S^3 \mathfrak{h} \right) $ is a real subspace of $d\rho (\mathfrak{h}) \otimes \mathbb{C}$.  This implies $\widetilde{d\rho} \left( S^3 \mathfrak{h} \right) \subseteq d\rho (\mathfrak{h})$.  In fact,  let $H_1,\dots,  H_n$ be a basis of $\mathfrak{h}$.  Then for each $X \in S^3 \mathfrak{h}$,  we may write $\widetilde{d\rho} (X) = \sum_{j=1}^n c_j d\rho (H_j)$ for some $c_j \in \mathbb{C}$.  Since $\widetilde{d\rho} (X)$,  $d \rho(H_1),\dots,  d \rho(H_n)$ are real,  we must have $\widetilde{d\rho} (X) = \sum_{j=1}^n \Re(c_j) d\rho (H_j) \in d\rho(\mathfrak{h})$.  Here $\Re(z)$ is the real part of $z\in \mathbb{C}$.
\item[$\diamond$] If $G$ is the compact form,  then $\lambda(H) \in i \mathbb{R} $ for any $\lambda \in \Lambda$ and $H\in \mathfrak{h}^{\mathbb{C}}$.  By \eqref{lem:rank-orbit:eq2} again,  $i \widetilde{d\rho} \left( S^3 \mathfrak{h} \right) $ is a real subspace of $d\rho (\mathfrak{h}) \otimes \mathbb{C}$.  This indicates that $\widetilde{d\rho} \left( S^3 \mathfrak{h} \right) \subseteq d\rho (\mathfrak{h})$.  \qedhere
\end{itemize}
\end{proof}

Although the geometric condition in Theorem~\ref{thm:main} is only a necessary condition for the applicability of the Cayley transform,  it is also sufficient if the group is compact.
\begin{proposition}[Characterization IV]\label{cond: Lie alg 3 TFAE}
Let $G$ be a compact semisimple Lie group and let $\rho: G \to \GL(\mathbb{V})$ be an irreducible representation such that $d\rho$ is faithful.  The followings are equivalent:
\begin{enumerate}[(a)]
\item $\rho$ is a Cayley representation.  \label{cond: Lie alg 3 TFAE:item1}
\item $ \widetilde{d\rho}\left( S^3 \mathfrak{h} \right) \subseteq d\rho (\mathfrak{h})$.  \label{cond: Lie alg 3 TFAE:item2}
\item $\rho$ is a representation with Cayley configuration.
 \label{cond: Lie alg 3 TFAE:item3}
\end{enumerate}
In particular,  if $G$ is compact and simple then \eqref{cond: Lie alg 3 TFAE:item1}--\eqref{cond: Lie alg 3 TFAE:item3} are equivalent for all irreducible representations of $G$.
\end{proposition}
\begin{proof}
Since $G$ is compact,  we have $\mathfrak{g}_0 = \mathfrak{h}$.  Moreover,  if $G$ is also simple then all its non-trivial representations are faithful.  By Proposition \ref{prop: compact form generates},  the equivalence between \eqref{cond: Lie alg 3 TFAE:item1} and \eqref{cond: Lie alg 3 TFAE:item2} follows from Theorem \ref{cond:simple}.  Thus,  it is left to establish the equivalence between \eqref{cond: Lie alg 3 TFAE:item2} and \eqref{cond: Lie alg 3 TFAE:item3}.  If \eqref{cond: Lie alg 3 TFAE:item2} holds,  then $\rho$ is Cayley and  Theorem~\ref{thm:main} leads to \eqref{cond: Lie alg 3 TFAE:item3}.  Conversely,  {we consider the universal covering $\pi: \widetilde{G}  \to G$ of $G$.  Then $\rho \circ \pi$ is a representation of $\widetilde{G}$ with Cayley configuration.  Proposition~\ref{prop: split compact back to real} implies that \eqref{cond: Lie alg 3 TFAE:item2} and hence \eqref{cond: Lie alg 3 TFAE:item1} holds for $\rho \circ \pi$.  By Proposition~\ref{prop: group to Lie algebra},  we conclude that $\rho$ is a Cayley representation.} 
\end{proof}

We notice that Proposition~\ref{cond: Lie alg 3 TFAE} is true for any semisimple Lie group $G$ such that $\mathfrak{g}_0 = \mathfrak{h}$ and $\mathfrak{g} = \Ad(G) (\mathfrak{h})$.  If $\mathbb{F} = \mathbb{C}$ or $G$ is a split real form,  then $\mathfrak{g}_0 = \mathfrak{h}$ always holds.  However,  the next proposition implies that it is impossible to have $\mathfrak{g} = \Ad(G) (\mathfrak{h})$ in these two cases.  
\begin{proposition}
If $G$ is either a complex semisimple Lie group or a split real form,  then there exists $x \in \mathfrak{g}$ such that $x \notin \Ad(G)(\mathfrak{h})$.
\end{proposition}
\begin{proof}
If  $G$ is either a complex semisimple Lie group or a split real form,  then we have the root space decomposition $\mathfrak{g} = \mathfrak{h} \bigoplus \left( \bigoplus_{\alpha \in \Phi} \mathfrak{g}_{\alpha} \right)$.  Let $\Ad: G \to \GL(\mathfrak{g})$ (resp.  $\ad: \mathfrak{g} \to \mathfrak{gl}(\mathfrak{g})$) be the adjoint representation of $G$ (resp.  $\mathfrak{g}$).  Given $\alpha \in \Phi$,  we pick a nonzero $X_{\alpha} \in \mathfrak{g}_{\alpha}$.  Since $\ad(X_{\alpha})(\mathfrak{g}_{\beta})  \subseteq \mathfrak{g}_{\alpha + \beta}$ for any $\beta \in \Phi$ and $\mathfrak{g}$ is finite dimensional,  there exists some integer $N \ge 0$ such that $\ad(X_{\alpha})^N (\mathfrak{g}_{\beta}) = 0$ for any $\beta \in \Phi$.  Therefore,  if $X_\alpha = \Ad(g)(H)$ for some $g \in G$ and $H\in \mathfrak{h}$,  then 
\[
0 = \ad(X_{\alpha})^N = \ad \left( \Ad(g)(H) \right)^N = \left( \Ad(g) \ad(H) \Ad(g)^{-1} \right)^N = \Ad(g) \ad(H)^N \Ad(g)^{-1}.
\]
This implies that $\ad(H) = 0$.  Since $\ad$ is injective,  we obtain $X_\alpha = H = 0$ which contradicts to the choice of $X_{\alpha}$. 
\end{proof}
For a concrete example,  we consider $G = \SL_2(\mathbb{F})$.  Then $\mathfrak{h}= \spn_{\FF} \{ H \}$ where $H = \begin{bsmallmatrix} 1 & 0 \\ 0 & -1 \end{bsmallmatrix}$.  It is obvious that $\Ad(\SL_2(\FF)) (\mathfrak{h}) \ne \mathfrak{sl}_2(\FF)$ as $X \not\in \Ad(\SL_2(\FF)) (\mathfrak{h})$ where $X = \begin{bsmallmatrix} 0 & 1 \\ 0 & 0 \end{bsmallmatrix}$.
\subsection{The Cayley transform on general representations}
In this subsection,  we discuss the applicability of the Cayley transform on an arbitrary representation of a semisimple Lie group $G$.  We recall that every finite dimensional representation $\rho: G \to \GL(\mathbb{V})$ admits a decomposition $\rho = \prod_{j=1}^m \rho_j: G \to \prod_{j=1}^m \GL(\mathbb{V}_j)$ where $\mathbb{V} = \bigoplus_{j=1}^m \mathbb{V}_j$ and for each $1 \le j \le m$,  $\rho_j: G \to \GL(\mathbb{V}_j)$ is an irreducible representation.
\begin{proposition}[Characterization V]\label{prop:reducible+faithful}
Let $G,  \mathbb{V}, \rho,  \rho_1,\dots,  \rho_m$ be as above.  The followings are equivalent: 
\begin{enumerate}[(a)]
\item $\rho$ is Cayley.    \label{prop:reducible+faithful:item1}
\item For each $1 \le j \le m$,  $\rho_j$ is Cayley.  \label{prop:reducible+faithful:item3}
\item There exists a semisimple Lie group $G'$ and a representation $\rho': G' \to \GL(\mathbb{V})$ such that $d\rho'$ is faithful,  $d\rho' (\mathfrak{g}') = d\rho (\mathfrak{g})$ and $\rho'$ is Cayley.   \label{prop:reducible+faithful:item2}
\item There exists a semisimple Lie group $G'$ and a faithful representation $\rho': G' \to \GL(\mathbb{V})$ such that $d\rho' (\mathfrak{g}') = d\rho (\mathfrak{g})$ and $\rho'$ is Cayley.   
\label{prop:reducible+faithful:item4}
\end{enumerate}
\end{proposition}
\begin{proof}
By Theorem~\ref{cond:3},  $\rho$ is Cayley if and only if $d \rho(\mathfrak{g})$ has the power span property.  The implication \eqref{prop:reducible+faithful:item2} $\implies$ \eqref{prop:reducible+faithful:item1} is obvious.  If \eqref{prop:reducible+faithful:item1} holds,  we let $G'$ be the simply connected Lie group such that $\mathfrak{g}' = d \rho(\mathfrak{g}) \subseteq \mathfrak{gl}(\mathbb{V})$ and let $\rho': G' \to \GL(\mathbb{V})$ be the corresponding representation of $G'$ on $\mathbb{V}$.  Then it is clear that $\rho'$ is Cayley,  thus we have \eqref{prop:reducible+faithful:item1} $\implies$ \eqref{prop:reducible+faithful:item2}.

Since $\rho = \bigoplus_{j=1}^m \rho_j$,  we have $d \rho (\mathfrak{g}) = \bigoplus_{j=1}^m d \rho_j (\mathfrak{g})$.  Therefore,  $d \rho(\mathfrak{g})$ has the power span property if and only if for each $1 \le j \le m$,  $d \rho_j (\mathfrak{g})$ has the power span property which is equivalent to the applicability of the Cayley transform to $\rho_j$,  by Theorem~\ref{cond:3} again.  This proves the equivalence between \eqref{prop:reducible+faithful:item1} and \eqref{prop:reducible+faithful:item3}.  

The implication \eqref{prop:reducible+faithful:item4} $\implies$ \eqref{prop:reducible+faithful:item2} is trivial.  Suppose $G_1$ and $\theta: G_1 \to \GL(\mathbb{V})$ satisfy conditions in \eqref{prop:reducible+faithful:item2}.  Since $d\theta$ is faithful,  $\ker(\theta)$ is a discrete normal subgroup of $G_1$.  Denote $G' \coloneqq G_1/\ker(\theta)$.  We define $\rho': G' \to \GL(\mathbb{V})$ by $\rho' \left( [g_1] \right) \coloneqq \theta(g_1)$ where $[g_1]$ denotes the element in $G'$ represented by $g_1 \in G_1$.  By definition,  $\rho'$ is well-defined and faithful.  We also observe that $\mathfrak{g}'=\mathfrak{g}_1$ and $d\rho' =d\theta$.  Therefore,  $\rho'$ is a representation satisfying \eqref{prop:reducible+faithful:item4}.
\end{proof}
According to Proposition~\ref{prop:reducible+faithful},  the applicability of the Cayley transform to $\rho: G\to \GL(\mathbb{V})$ reduces to the applicability of the Cayley transform to $\rho_j': G' \to \GL(\mathbb{V}_j)$ where $\rho' = \bigoplus_{j=1}^m \rho'_j$,  $\mathbb{V} = \bigoplus_{j=1}^m \mathbb{V}_j$ and for each $1 \le j \le m$,  $\rho'_j: G' \to \GL(\mathbb{V}_j)$ is an irreducible faithful representation.  Moreover,  we observe that $\widetilde{d\rho} \left( S^3 \mathfrak{h} \right) \subseteq d \rho (\mathfrak{h})$ if and onlyf if $\widetilde{d\rho_j} \left( S^3 \mathfrak{h} \right) \subseteq d \rho_j (\mathfrak{h})$,  $1 \le j \le m$.  This together with Theorem~\ref{thm:main} and Proposition~\ref{cond: Lie alg 3 TFAE}  leads to the proposition that follows.
\begin{proposition}[Characterization VI]\label{prop:Characterization VI}
Let $G,  G'$,  $\mathbb{V},  \mathbb{V}_1,\dots,  \mathbb{V}_m$ and $\rho,  \rho', \rho'_1,  \dots,  \rho'_m$ be as above.  The followings hold:
\begin{enumerate}[(a)]
\item If $\mathfrak{g}_0 = \mathfrak{h}$ and $\rho$ is Cayley,  then for each $1 \le j \le m$,  $\rho'_j$ is a representation with Cayley configuration.
\item For a compact semisimple Lie group $G$,  the followings are equivalent:
\begin{enumerate}[(i)]
\item $\rho$ is Cayley. 
\item $\widetilde{d\rho} \left( S^3 \mathfrak{h} \right) \subseteq d \rho (\mathfrak{h})$.
\item For each $1 \le j \le m$,  $\rho'_j$ is a representation with Cayley configuration.
\end{enumerate}

\end{enumerate} 
\end{proposition}
\section{The Cayley transform for simple Lie groups}\label{sec:simple}
In this section,  we consider the applicability of the Cayley transform to a representation $\rho: G \to \GL(\mathbb{V})$ of a \emph{classical simple} Lie group $G$.  In fact,  one can prove that the Cayley transform is not applicable to \emph{exceptional} simple Lie groups by similar calculations in Subsection~\ref{subsec:CayleyRep}.  This is also a direct consequence of \cite[Theorem 1.31]{LPR06}.

We notice that the Lie algebra $\mathfrak{g}$ of $G$ is a simple Lie algebra,  thus $d \rho:  \mathfrak{g} \to \mathfrak{gl}(\mathbb{V})$ is faithful unless $d\rho (\mathfrak{g}) = \{ 0 \}$.  The latter case indicates that $\rho$ is the trivial representation.  In the sequel,  we suppose $\rho$ is non-trivial so that $d\rho$ is faithful.  
\subsection{Representations with Cayley configuration}\label{subsec:CayleyRep}
We first classify Cayley representions of \emph{classical complex simple} Lie groups.  It is well-known \cite[Theorem~21.11]{fulton2013representation} that they are classified into four types: $A_n,  B_n, C_n$ and $D_n$.  Hence we may split our discussion correspondingly.  For ease of reference,  we record in Table~\ref{tab:simple} the information of the root system for each type.  Here $\mathfrak{h}$ is the Cartan subalgebra of $\mathfrak{g}$,  $\Lambda$ is the weight lattice,  $\Phi^+$ is the set of positive roots,  $\Delta^+$ is the set of simple roots,  $C = \{ \sum_{j=1}^n a_j L_j \}$ is the fundamental (closed) Weyl chamber,  $\W$ is the Weyl group and $\delta \coloneqq \frac{1}{2} \sum_{j=1}^n L_j $.  Moreover,  in the last line of Table~\ref{tab:simple},  we also list the standard representation $\mathbb{W}$ of each type together with its weights.
\begin{table}[h]
\footnotesize
\tabulinesep=0.5ex
\begin{tabu}{l|l|l|l|l}
\textsc{type} & $A_n$ ($\mathfrak{sl}_{n+1}$, $n \ge 1$) &  $B_n$ ($\mathfrak{so}_{2n+1}$, $n \ge2$) &  $C_n$ ($\mathfrak{sp}_{2n}$, $n\ge 3$) &  $D_n$ ($\mathfrak{so}_{2n}$,  $n\ge 4$)  \\\hline\hline
$\dim \mathfrak{h}$ &  $n$ & $n$   & $n$   &  $n$ \\\hline
$\Lambda$ & $\mathbb{Z}\{ L_j: \sum_{j=1}^{n+1} L_j=0 \}_{j=1}^{n+1}$ & $\mathbb{Z}\{ \delta,  L_j\}_{j=1}^{n}$ &     $\mathbb{Z}\{ L_j\}_{j=1}^{n}$ & $\mathbb{Z}\{\delta,  L_j\}_{j=1}^{n}$  \\\hline
$\Phi^+$ & $L_j - L_k$,\; $j < k$ & $L_i$,  \; $L_j \pm L_k$,\; $j < k$ & $2 L_i$,  $L_j \pm L_k$,\; $j < k$ & $ L_j \pm L_k$ ,\; $j < k$ \\\hline
$\Delta^+$ & $L_j - L_{j+1}$ & $L_n$,  \; $L_j - L_{j+1}$ & $2 L_n$,  \; $L_j - L_{j+1}$ &$L_{n - 1} + L_n$, \; $ L_j - L_{j+1}$ \\\hline
$C$ & $a_{j} \ge a_{j+1}$,\; $a_{n+1} = 0$ &  $a_j \ge a_{j+1} \ge 0$ & $a_j \ge a_{j+1} \ge 0$ & $a_1 \ge \cdots \ge  a_{n-1} \ge |a_n|$   \\\hline
$\W$ & $\mathfrak{S}_{n+1}$ &  $ \W/ (\mathbb{Z}/2\mathbb{Z})^n \simeq \mathfrak{S}_n$ &   $ \W/ (\mathbb{Z}/2\mathbb{Z})^n \simeq \mathfrak{S}_n$  & $ \W/ (\mathbb{Z}/2\mathbb{Z})^{n-1} \simeq \mathfrak{S}_n$ \\\hline
$\mathbb{W}$ & $\mathbb{C}^{n+1}$ (weights: $L_j$) & $\mathbb{C}^{2n+1}$ (weights: $0$,  $\pm L_j$) &   $\mathbb{C}^{2n}$ (weights: $\pm L_j$)  & $\mathbb{C}^{2n}$ (weights: $\pm L_j$)\\\hline
\end{tabu}
\bigskip
\caption{Root system of a complex simple Lie group}
\label{tab:simple}
\end{table}

\subsubsection{Type $A_n$ ($ n \ge 1$)}

\begin{lemma}\label{prop:adjoint}
Let $G$ be a complex simple Lie group of type $A_n$ and let $\rho: G \to \GL (\mathbb{V})$ be a representation with Cayley configuration.  If $\mathbb{V}_0 \ne 0$,  then $n = 1$ and $\rho$ is the adjoint representation.
\end{lemma}
\begin{proof}
By definition,  $S \coloneqq \{\lambda \in \Lambda: \mathbb{V}_{\lambda} \ne 0 \} = \mathcal{O} \bigsqcup \{0\}$ where $\mathcal{O} \coloneqq \W \omega_1 = \{\pm \omega_1,\dots,  \pm \omega_n\}$ has rank $n$ and $\omega_1$ is the highest weight of $\mathbb{V}$.  Since $S = (\omega_1 + \mathbb{Z} \Phi) \cap \Conv(\mathcal{O})$,  we must have $-\omega_1 \in \mathbb{Z} \Phi$ and hence $\omega_1 \in \mathbb{N} \Delta^+$.  We notice that $\mathcal{O}$ is the $\W$-orbit of the highest weight $\omega_1$.  Thus,  elements in $\mathcal{O}$ belong to different Weyl chambers.  This implies that $0$ and $\omega_1$ are the only weights of $\mathbb{V} $ contained in the fundamental Weyl chamber $C$.  Therefore,  we may conclude that $\omega_1 \in \Delta^+$.  According to Table~\ref{tab:simple},  $\W$ acts transitively on $\Phi$,  from which we obtain $\mathcal{O} = \Phi$.  Since both $\rho$ and the adjoint representation are irreducible,  we conclude that they are isomorphic.  By the definition of the adjoint representation,  we have $(n+1)^2 - (n+1) = \dim (\mathfrak{g}) - \dim (\mathfrak{h}) = |\Phi| = |\mathcal{O}| = 2n$ which implies $n = 1$.
\end{proof}

\begin{lemma}
Let $G$ be a complex simple Lie group of type $A_1$ and let $\rho: G \to \GL (\mathbb{V})$ be a representation with Cayley configuration.  Then $\rho$ is either the standard representation or the adjoint representation.
\end{lemma}
\begin{proof}
Let $\mathbb{C}^2$ be the standard representation of $G \simeq \SL_2(\mathbb{C})$.  Then any irreducible representation of $G$ is isomorphic to $\mathsf{S}^d \mathbb{C}^2$ for some positive integer $d$. Since the set of weights of $\mathsf{S}^d \mathbb{C}^2$ is $\{d - 2k\}_{k=0}^d$,  it is straightforward to verify that $\mathsf{S}^d \mathbb{C}^2$ is a representation with Cayley configuration if and only if $d = 1$ (the standard representation) or $d = 2$ (the adjoint representation). 
\end{proof}
    
\begin{lemma}\label{lem:A_2}
If $G$ is a complex simple Lie group of type $A_2$,  then $G$ has no representation with Cayley configuration.
\end{lemma}
\begin{proof}
By Table~\ref{tab:simple},  we have $\dim (\mathfrak{h})= 2$ and $\W = \mathfrak{S}_3$. Since the cardinality of any $\W$-orbit divides $|\W| = 6$,  there does not exist an orbit of cardinality $4 = 2 \dim (\mathfrak{h})$.  This implies the non-existence of a representation with Cayley configuration.  
\end{proof}

\begin{lemma}\label{lem:An}
If $G$ is a complex simple Lie group of type $A_n$ and $\rho: G \to \GL (\mathbb{V})$ is a representation with Cayley configuration,  then one of the followings holds: 
\begin{enumerate}[(a)]
\item $n = 1$ and $\mathbb{V}$ is the standard representation $\mathbb{C}^2$. 
\item $n = 1$ and $\mathbb{V}$ is the adjoint representation $\mathfrak{sl}_2(\mathbb{C})\simeq \mathsf{S}^2 \mathbb{C}^2$.  
\item $n =3$ and $\mathbb{V} = \mathsf{\Lambda}^2 \mathbb{C}^4$.
\end{enumerate}
\end{lemma}
\begin{proof}
By Lemmas~\ref{prop:adjoint}--\ref{lem:A_2},  we may assume that $0$ is not a weight of $\mathbb{V}$ and $n \ge 3$.  Let $\omega_1$ be the highest weight of $\mathbb{V}$.  We may write $\omega_1 = \sum_{j=1}^n a_j L_j$ for some integers $a_1 \ge \cdots \ge a_n \ge 0$.  Suppose $\sigma \in \W$ is the reflection across the hyperplane perpendicular to $L_1 - L_{n+1}$.  Then $\sigma$ interchanges $L1$ and $L_{n+1}$ and fixes $L_2,\dots, L_n$.  Hence we have 
\[
\sigma (\omega_1) = \sum_{j=2}^{n} a_j L_j + a_1 L_{n+1},\quad \sigma(\omega_1) - \omega_1 = a_1 (L_{n+1} - L_1).
\] 
Since $L_{n+1} - L_1$ is a root and both $\omega_1$ and $\sigma(\omega_1)$ are weights of $\mathbb{V}$,  $\omega_1 + k (L_{n+1} - L_1)$ is also a weight of $\mathbb{V}$ for $0 \le k \le a_1$.  By the definition of a representation with Cayley configuration,  we must have $a_1 = 1$.  Hence there exists some positive integer $k \le n$ such that $\omega_1 = \sum_{j=1}^k L_j$.  This implies that 
\[
\W \omega_1 = \left\lbrace \sum_{j=1}^k \omega_{s_j}: 1 \le s_j \le n+1,\;s_j \ne s_{j'},\;1 \le j \ne j' \le k\right\rbrace.
\]
Therefore,  we have $\binom{n+1}{k} = |\W \omega_1| = 2 \dim(\mathfrak{h}) = 2n$,  which forces $(n,k) = (3,2)$ as $n \ge 3$.  Now that $G \simeq \SL_4(\mathbb{C})$ and $\omega_1 = L_1 + L_2$,  we have $\mathbb{V} = \mathsf{\Lambda}^2 \mathbb{C}^4$.
\end{proof}
\subsubsection{Type $B_n$ ($n \ge 2$)}
\begin{lemma}\label{lem:adjointBn}
Let $G$ be a complex simple Lie group of type $B_n$ ($n\ge 2$) and let $\rho: G \to \GL(\mathbb{V})$ be a representation with Cayley configuration such that $\mathbb{V}_0 \ne 0$.  Then $\rho$ is the standard representation.
\end{lemma}
\begin{proof}
By definition,  $S \coloneqq \{\lambda \in \Lambda: \mathbb{V}_{\lambda} \ne 0 \} = \mathcal{O} \bigsqcup \{0\}$ where $\mathcal{O} \coloneqq \W \omega_1 = \{\pm \omega_1,\dots,  \pm \omega_n\}$ has rank $n$ and $\omega_1$ is the highest weight of $\mathbb{V}$.  By the same argument as in the proof of Lemma~\ref{prop:adjoint},  we obtain that $\omega_1 \in \Delta^+$.  According to Table~\ref{tab:simple},  we have \[
\Phi = \W L_n \bigsqcup \W ( L_1 - L_2),
\]
from which we may conclude that either $\omega_1 =  L_n$ or $\omega_1 =  L_1 - L_2$.  By the irreducibility of $\mathbb{V}$,  $\rho$ is either the standard representation or the adjoint representation.  To exclude the latter,  we notice that $n \ge 2$ implies
\[ 
\dim (\mathfrak{g}) - \dim (\mathfrak{h}) = n(2n+1) - n  > 2n =  2\dim (\mathfrak{h}).  \qedhere
\]
\end{proof}
\begin{lemma}\label{lem:Bn}
Let $G$ be a complex simple Lie group of type $B_n$ ($n\ge 2$) and let $\rho: G \to \GL(\mathbb{V})$ be a representation with Cayley configuration.  Then $\rho$ is either the standard representation or $n = 2$ and $\rho$ is the spin representation.
\end{lemma}
\begin{proof}
By Lemma~\ref{lem:adjointBn},  it is sufficient to consider the case $\mathbb{V}_0 = 0$.  Let $\omega_1$ be the highest weight of $\mathbb{V}$.  We may write $\omega_1 = \frac{1}{2} \sum_{j=1}^n a_j L_j $ from some integers $a_1 \ge \cdots a_n \ge 0$ of the same parity.  Let $\sigma \in \W$ be the reflection such that $\sigma(L_1) = -L_1$ and $\sigma(L_j) = L_j$ if $2 \le j \le n$.  We have $\sigma (\omega_1) - \omega_1 = -a_1 L_1$.  The same argument as in the proof of Lemma~\ref{lem:An} implies $a_1 = \cdots = a_n = 1$ and $|\W \omega_1| = 2^n >  2n = 2 \dim (\mathfrak{h})$,  which contradicts to the assumption that $\mathbb{V}$ is a representation with Cayley configuration when $n \ge 3$.  If $n = 2$,  then $\omega_1 = \frac{1}{2}( L_1 + L_2 )$ and $\mathbb{V}$ is the spin representation by definition.
\end{proof}    
\subsubsection{Type $C_n$ ($n \ge 3$)}    
By the argument in the proof of Lemma~\ref{lem:adjointBn},  we may derive that $0$ is not a weight of a representation with Cayley configuration for simple Lie group of type $C_n$.
\begin{lemma}\label{lem:adjointCn}
Let $G$ be a complex simple Lie group of type $C_n$ ($n\ge 3$) and let $\rho: G \to \GL(\mathbb{V})$ be a representation with Cayley configuration.  Then we have $\mathbb{V}_0 = 0$.   
\end{lemma}
Let $\mathbb{V}$ be a representation of $G$ with Cayley configuration and let $\omega_1$ be its highest weight.  A similar calculation as in the proof of Lemma~\ref{lem:Bn} leads to $\W \omega = \{ \sum_{j=1}^k \pm L_{s_j}: 1 \le s_1 < \cdots < s_k  \le n\}$ for some $1 \le k \le n$.  Since $0$ is not a weight of $\mathbb{V}$,  we have 
\[
2^k \binom{n}{k} = |\W \omega| = 2\dim (\mathfrak{h}) = 2n.
\]
Since $n \ge 3$,  we obtain that $k = 1$ which indicates the following.
\begin{lemma}\label{lem:C_n}
Let $G$ be a complex simple Lie group of type $C_n$ ($n\ge 3$) and let $\rho: G \to \GL(\mathbb{V})$ be a representation with Cayley configuration.  Then $\rho$ is the standard representation. 
\end{lemma}
\subsubsection{Type $D_n$ ($n \ge 4$)} Repeating the argument for $B_n$ and $C_n$,  we may prove that the set of weight of a representation with Cayley configuration is a $\W$-orbit.
\begin{lemma}
\label{lem:adjointDn}
Let $G$ be a  complex simple Lie group of type $C_n$ ($n\ge 3$) and let $\rho: G \to \GL(\mathbb{V})$ be a representation with Cayley configuration.  Then we have $\mathbb{V}_0 = 0$.   
\end{lemma}

We recall that a \emph{spinorial representation} is an irreducible representation $\rho: G \to \GL(\mathbb{V})$ with the highest weight $\omega = \frac{1}{2}\sum_{j=1}^{n} a_j L_i$ where $a_1 \ge \cdots \ge a_{n-1} \ge |a_n|$ are odd.  Moreover,   a \emph{spinor} representations is an irreducible representation with the highest weight $\omega_+ \coloneqq \frac{1}{2} \sum_{j=1}^n L_j$ or $\omega_- \coloneqq \left( \frac{1}{2}  \sum_{j=1}^{n} L_j \right) - L_n$.
\begin{lemma}\label{lem:spinorDn}
Let $G$ be a complex simple Lie group of type $D_n$ ($n\ge 4$) and let $\rho: G \to \GL(\mathbb{V})$ be a representation with Cayley configuration.  If $\rho$ is spinorial,  then $n = 4$ and $\rho$ is a spinor representation. 
\end{lemma}
\begin{proof}
We first prove that $\rho$ is a spinor representation.  We only sketch the proof since it is similar to that of Lemma~\ref{lem:Bn}.  Let $\omega_1 = \frac{1}{2}\sum_{j=1}^{n} a_j L_j$ be the highest weight of $\rho$,  where $a_1 \ge \cdots \ge a_{n-1} \ge |a_n| \ge 0$ are some odd integers.  For $2 \le j \le n$,  we let $\sigma_j \in \W$ be the reflection that fixes $L_k$ for $k\ne 1, j$ and interchanges $L_1$ and $L_j$.  Then we have $\sigma_j (\omega_1) - \omega_1 = \frac{1}{2}(a_j - a_1)(L_1 - L_j)$.  Hence we may obtain that $a_1 - a_j \le 2$.  On the other side,  we let $\tau_j$ be the reflection such that 
\[
\tau_j(L_1) = -L_j,\; \tau_j(L_j) = -L_1,\; \tau_j (L_k) = L_k,\quad 2 \le k \ne j \le n.
\]
We obtain $\tau_j(\omega_1) - \omega_1 =- \frac{1}{2}(a_1+ a_j)(L_1 + L_j)$ from which $a_1 + a_j \le 2$.  This implies $a_k = 1$ if $1 \le k \le n-1$ and $a_n = \pm 1$.  Hence $\rho$ is a spinor representation.  

Next we prove that $n = 4$.  By Table~\ref{tab:simple},  we have 
\begin{align*}
\W \omega_+ &= \left\lbrace
\frac{1}{2} \sum_{j=1}^{n} a_j L_j: |a_j| = 1,\; \sum_{j=1}^n a_j \equiv n \pmod{4}, \; 1\le j \le n 
\right\rbrace,\\ 
\W \omega_- &= \left\lbrace
\frac{1}{2} \sum_{j=1}^{n} a_j L_j: |a_j| = 1,\; \sum_{j=1}^n a_j \not\equiv  n \pmod{4},  \; 1 \le j \le n 
\right\rbrace.
\end{align*}
Since $\mathbb{V}$ is a representation with Cayley configuration and spinorial,  we have either $|\W \omega_+| = 2 \dim (\mathfrak{h})$ or $|\W \omega_-| = 2 \dim (\mathfrak{h})$.  This forces $2^{n-1}=2n$ and hence $n = 4$.  

Lastly,  we notice that if $n = 4$ then
\begin{align*}
\W \omega_+ &= \left\lbrace
\frac{1}{2} \sum_{j=1}^4 a_j L_j: |a_j| = 1, \; \sum_{j=1}^4 a_j \in \{-4,  0,4\}, \;  1 \le j \le 4
\right\rbrace,  \\
W \omega_- &= \left\lbrace
\frac{1}{2} \sum_{j=1}^4 a_j L_j: |a_j| = 1, \; \sum_{j=1}^4 a_j \in \{-3,-1, 1,3\}, \;  1 \le j \le 4
\right\rbrace.  
\end{align*}
It is straightforward to verify that such spinor representations are representations with Cayley configuration.
\end{proof}

Now we are able to determine all representations with Cayley configuration.
\begin{lemma}\label{lem:Dn}
Let $G$ be a simply-connected complex simple Lie group of type $D_n$ ($n\ge 4$) and let $\rho: G \to \GL(\mathbb{V})$ be a representation with Cayley configuration.  Then either $\rho$ is the standard representation or $n =4$ and $\rho$ is one of the two spinor representations.
\end{lemma}
\begin{proof}
By Lemma~\ref{lem:spinorDn},  it suffices to consider irreducible representations that are not spinorial.  The rest of the proof is similar to that of Lemma~\ref{lem:spinorDn}.  We write the highest weight $\omega_1$ of $\mathbb{V}$ as $\omega_1 = \sum_{j=1}^n a_j L_j$ where $a_1 \ge \cdots \ge a_{n-1} \ge |a_n| \ge 0$ are some integers.  For $2 \le j \le n$,  we may again choose different reflections to prove that $a_1 \pm a_j \le 1$,  which implies that $a_1 = 1$ and $a_j = 0$ for each $2 \le j \le n$.  According to Lemma~\ref{lem:adjointDn},  $\rho$ is the standard representation.
\end{proof}
\subsection{Classification}
In Subsection~\ref{subsec:CayleyRep}, we classify all representations with Cayley configuration of classical complex simple Lie groups.  In this Subsection,  we further classify Cayley representations for classical complex simple Lie groups and their compact real forms.  To achieve the goal,  we first recall isomorphisms among complex simple Lie algebras of small dimension.
\begin{remark}\label{rmk:small}
We recall that for small dimensional complex simple Lie algebras,  there are some isomorphisms among them: 
\begin{enumerate}[(a)]
\item $\mathfrak{sl}_2(\CC)\simeq \mathfrak{so}_3(\CC)  \simeq \mathfrak{sp}_2(\CC)$.
\item $\mathfrak{sl}_4(\CC)\simeq \mathfrak{so}_6(\CC)$.
\item $\mathfrak{so}_5(\CC) \simeq \mathfrak{sp}_4(\CC)$.
\end{enumerate}
Moreover,  $\mathfrak{so}_2(\CC) \simeq \mathbb{C}$ and $\mathfrak{so}_4(\CC)\simeq \mathfrak{sl}_2(\CC) \times \mathfrak{sl}_2(\CC)$ are not simple.
\end{remark}

\begin{theorem}[Classification I]\label{thm:classification1}
Let $G$ be a classical complex simple Lie group and let $\rho: G \to \GL(\mathbb{V})$ be an irreducible representation.  Suppose $\mathfrak{g}$ is the Lie algebra of $G$.  Then the followings are equivalent
\begin{enumerate}[(a)]
\item $\rho$ is a representation with Cayley configuration. \label{thm:classification1 item1}
\item $\rho$ is a Cayley representation.\label{thm:classification1 item2}
\item One of the followings holds: \label{thm:classification1 item3}
\begin{enumerate}[(i)]
\item $\mathfrak{g} \simeq \mathfrak{sl}_2(\CC)$ and $\mathbb{V} \simeq \mathbb{C}^2$.  Under the identification $\mathfrak{sl}_2(\CC) \simeq \mathfrak{sp}_2(\CC)$,  $\mathbb{V}$ is isomorphic to the standard representation.\label{eg: sl2 standard}
\item $\mathfrak{g} \simeq \mathfrak{sl}_2(\mathbb{C})$ and $\mathbb{V} \simeq \mathfrak{sl}_2(\mathbb{C})$.  Under the identification $\mathfrak{sl}_2(\CC) \simeq \mathfrak{so}_3(\CC)$,  we have $\mathbb{V} \simeq \mathbb{C}^3$. \label{eg: sl2 adjoint}
\item $\mathfrak{g} \simeq \mathfrak{sl}_4(\CC)$ and $\mathbb{V} \simeq \mathsf{\Lambda}^2 \mathbb{C}^4$.  Under the identification $\mathfrak{sl}_4(\mathbb{C}) \simeq \mathfrak{so}_6(\mathbb{C})$,  we have $\mathbb{V}\simeq \mathbb{C}^6$.\label{eg: sl4}
\item $\mathfrak{g} \simeq \mathfrak{so}_{2n+1}(\CC)$ ($n \ge 2$) and $\mathbb{V} \simeq \mathbb{C}^{2n+1}$.  \label{eg: so standard}
\item $\mathfrak{g} \simeq \mathbb{S})$ and $\mathbb{V}$ is the spin representation.  Under the identification $\mathfrak{so}_5(\CC) \simeq \mathfrak{sp}_4(\CC)$,  we have $\mathbb{V} \simeq \mathbb{C}^4$. \label{eg: so5 spin}
\item $\mathfrak{g} \simeq \mathfrak{sp}_{2n}(\CC)$ ($n \ge 3$) and $\mathbb{V} \simeq \mathbb{C}^{2n}$.  \label{eg: sp standard}
\item $\mathfrak{g} \simeq \mathfrak{so}_{8}(\CC)$ and $\mathbb{V}$ is one of the two spinor representations. \label{eg: so8 spin}
\item $\mathfrak{g} \simeq \mathfrak{so}_{2n}(\CC)$ ($n \ge 4$) and $\mathbb{V} \simeq \mathbb{C}^{2n}$.\label{eg: so standard}
\end{enumerate}
\item {For any $x \in \mathfrak{g}$,  we have $C(d\rho(x)) \in \rho(G)$ as long as  $1 - d\rho(x)$ is invertible. } \label{thm:classification1 item4}
\end{enumerate}
\end{theorem}
\begin{proof}
The equivalence between \eqref{thm:classification1 item1} and \eqref{thm:classification1 item3} follows  immediately from Lemmas~\ref{lem:An},  \ref{lem:Bn},  \ref{lem:C_n},  \ref{lem:Dn} and Remark \ref{rmk:small}. The implication \eqref{thm:classification1 item2} $\implies$ \eqref{thm:classification1 item1} is a consequence of Theorem~\ref{thm:main}.  For the converse,  we notice that except for \eqref{eg: so8 spin},  all cases in \eqref{thm:main:item3} correspond to standard representations of quadratic matrix groups.  Thus,  they are Cayley representations.  By Remark~\ref{rmk:spin8},  \eqref{eg: so8 spin} is also Cayley and this completes the proof of \eqref{thm:classification1 item1} $\implies$ \eqref{thm:classification1 item2}.  {Lastly,  \eqref{thm:classification1 item4} implies \eqref{thm:classification1 item1} is clear by definition and it is straightforward to verify the implication  \eqref{thm:classification1 item3} $\implies$ \eqref{thm:classification1 item4}. }
\end{proof}
\begin{remark}
We observe that \eqref{thm:classification1 item4} is a global (and stronger) version of \eqref{thm:classification1 item1}.  However,  Theorem~\ref{thm:classification1} reveals that the two properties are equivalent for classical complex simple Lie groups.
\end{remark}

Next we classify Cayley representations of $G$ when $G$ is compact.
\begin{theorem}[Classification II]\label{thm:classification2}
Let $G$ be a compact real simple Lie group and let $\rho: G \to \GL(\mathbb{V})$ be an irreducible representation of $G$.  The Cayley transform is applicable to $\rho$ if and only if $(\mathfrak{g},  \mathbb{V})$ belongs to the following list 
\begin{enumerate}[(a)]
    \item $(\mathfrak{su}_2(\CC),  \mathbb{C}^2)$.  Under the identification $\mathfrak{su}_2(\CC) \simeq \mathfrak{so}_3(\RR)$,  $\mathbb{V}$ is isomorphic to the spinor representation. 
    \item $(\mathfrak{su}_2(\mathbb{C}), \mathfrak{sl}_2(\CC))$,  where $\mathfrak{sl}_2(\CC)$ is the complexification of the adjoint representation of $\mathfrak{su}_{2}(\CC)$.  Under the identification $\mathfrak{su}_2(\CC) \simeq \mathfrak{so}_3(\RR)$,  we have $\mathbb{V} \simeq \mathbb{C}^3$.
    \item $(\mathfrak{su}_4(\CC), \mathsf{\Lambda}^2 \mathbb{C}^4)$.  Under the identification $\mathfrak{su}_4(\mathbb{C}) \simeq \mathfrak{so}_6(\mathbb{R})$,  we have $\mathbb{V}\simeq \mathbb{C}^6$.
    \item $(\mathfrak{so}_{2n+1}(\RR),  \mathbb{C}^{2n+1})$ ($n \ge 2$).
    \item $(\mathfrak{so}_5(\mathbb{R}),  \mathbb{S})$. $\mathbb{S}$ is the spinor representation.  Under the identification $\mathfrak{so}_5(\RR) \simeq \mathfrak{u}_2(\HH)$,  we have $\mathbb{V} \simeq \mathbb{H}^2$. Here $\mathbb{H}$ denotes the quaternion algebra.
    \item $(\mathfrak{u}_{n}(\HH),  \mathbb{H}^{n})$ ($n \ge 3$). 
    \item $(\mathfrak{so}_{8}(\RR),  \mathbb{S}_+)$ or $(\mathfrak{so}_{8}(\RR),  \mathbb{S}_-)$,  where $\mathbb{S}_+$ and $\mathbb{S}_-$ are the two spinor representations. 
    \item $(\mathfrak{so}_{2n}(\RR),  \mathbb{C}^{2n} )$ ($n \ge 4$).
\end{enumerate}
\end{theorem}
\begin{proof}
The list can be easily obtained by Proposition~\ref{prop:Characterization VI},  Lemmas~\ref{lem:An},  \ref{lem:Bn},  \ref{lem:C_n},  \ref{lem:Dn},  Remarks~\ref{rmk:spin8} and \ref{rmk:small}.
\end{proof}
The following is a direct consequence of Theorems~\ref{thm:classification1} and \ref{thm:classification2}.
\begin{corollary}
Let $G$ be a simple Lie group over $\mathbb{F}$ such that $\mathfrak{g}^{\mathbb{C}} \ne \mathfrak{spin}_8(\mathbb{C})$ ($\mathbb{F} = \mathbb{R}$) and $\mathfrak{g} \ne \mathfrak{spin}_8(\mathbb{C})$ ($\mathbb{F} = \mathbb{C}$).  Then representations with Cayley configuration of $G$ are fixed by $\Out(G)$.  If $(\mathbb{F},  \mathfrak{g}^{\mathbb{C}}) = (\mathbb{R},  \mathfrak{spin}_8(\mathbb{C}))$ or $(\mathbb{F},  \mathfrak{g}) = (\mathbb{C},  \mathfrak{spin}_8(\mathbb{C}))$,  then Caley representations of $G$ are $\mathbb{C}^8,  \mathbb{S}_{+}$ and $\mathbb{S}_{-}$ which are permuted by $\Out(G) \simeq \mathfrak{S}_3$.  
\end{corollary}

\bibliographystyle{plain}
\bibliography{cayley_ref}

\begin{thebibliography}{10}

\bibitem{ACD04}
F.~Astengo, M.~Cowling, and B.~Di~Blasio.
\newblock The {C}ayley transform and uniformly bounded representations.
\newblock {\em J. Funct. Anal.}, 213(2):241--269, 2004.

\bibitem{AD06}
Francesca Astengo and Bianca Di~Blasio.
\newblock Sobolev spaces and the {C}ayley transform.
\newblock {\em Proc. Amer. Math. Soc.}, 134(5):1319--1329, 2006.

\bibitem{BCM04}
Zheng-Jian Bai, Raymond~H. Chan, and Benedetta Morini.
\newblock An inexact {C}ayley transform method for inverse eigenvalue problems.
\newblock {\em Inverse Problems}, 20(5):1675--1689, 2004.

\bibitem{BG96}
George~A. Baker, Jr. and Peter Graves-Morris.
\newblock {\em Pad\'{e} approximants}, volume~59 of {\em Encyclopedia of Mathematics and its Applications}.
\newblock Cambridge University Press, Cambridge, second edition, 1996.

\bibitem{BR85}
Peter Bardsley and R.~W. Richardson.
\newblock \'{E}tale slices for algebraic transformation groups in characteristic {$p$}.
\newblock {\em Proc. London Math. Soc. (3)}, 51(2):295--317, 1985.

\bibitem{Borovoi15}
Mikhail Borovoi.
\newblock Real reductive {C}ayley groups of rank 1 and 2.
\newblock {\em J. Algebra}, 436:35--60, 2015.
\newblock With appendices by Dolgachev and an anonymous referee.

\bibitem{BK15}
Mikhail Borovoi and Boris Kunyavski$\breve{i}$.
\newblock Stably {C}ayley semisimple groups.
\newblock {\em Doc. Math.}, (Extra vol.: Alexander S. Merkurjev's sixtieth birthday):85--112, 2015.

\bibitem{BKLR14}
Mikhail Borovoi, Boris Kunyavski$\breve{i}$, Nicole Lemire, and Zinovy Reichstein.
\newblock Stably {C}ayley groups in characteristic zero.
\newblock {\em Int. Math. Res. Not.}, (19):5340--5397, 2014.

\bibitem{bourbaki2006groupes}
Nicolas Bourbaki.
\newblock {\em Groupes de Lie}.
\newblock Springer, 2006.

\bibitem{BKR20}
Chris Bourne, Johannes Kellendonk, and Adam Rennie.
\newblock The {C}ayley transform in complex, real and graded {$K$}-theory.
\newblock {\em Internat. J. Math.}, 31(9):2050074, 50, 2020.

\bibitem{Cayley1846}
A.~Cayley.
\newblock Sur quelques propri\'et\'es des d\'eterminants gauches.
\newblock {\em J. Reine Angew. Math.}, 32:119--123, 1846.

\bibitem{CDKR98}
Michael Cowling, Anthony Dooley, Adam Kor\'{a}nyi, and Fulvio Ricci.
\newblock An approach to symmetric spaces of rank one via groups of {H}eisenberg type.
\newblock {\em J. Geom. Anal.}, 8(2):199--237, 1998.

\bibitem{CDK91}
Michael Cowling, Anthony~H. Dooley, Adam Kor\'{a}nyi, and Fulvio Ricci.
\newblock {$H$}-type groups and {I}wasawa decompositions.
\newblock {\em Adv. Math.}, 87(1):1--41, 1991.

\bibitem{ESAH20}
S.~Emura, H.~Sawada, S.~Araki, and N.~Harada.
\newblock A frequency-domain bss method based on l1 norm, unitary constraint, and cayley transform.
\newblock In {\em ICASSP 2020 - 2020 IEEE International Conference on Acoustics, Speech and Signal Processing (ICASSP)}, pages 111--115, 2020.

\bibitem{fulton2013representation}
William Fulton and Joe Harris.
\newblock {\em Representation theory}, volume 129 of {\em Graduate Texts in Mathematics}.
\newblock Springer-Verlag, New York, 1991.
\newblock A first course, Readings in Mathematics.

\bibitem{GSAS21}
Bin Gao, Nguyen~Thanh Son, P.-A. Absil, and Tatjana Stykel.
\newblock Riemannian optimization on the symplectic {S}tiefel manifold.
\newblock {\em SIAM J. Optim.}, 31(2):1546--1575, 2021.

\bibitem{GMP11}
A.~G\'omez-Tato, E.~Mac\'ias-Virg\'os, and M.~J. Pereira-S\'aez.
\newblock Trace map, {C}ayley transform and {LS} category of {L}ie groups.
\newblock {\em Ann. Global Anal. Geom.}, 39(3):325--335, 2011.

\bibitem{HLW10}
Ernst Hairer, Christian Lubich, and Gerhard Wanner.
\newblock {\em Geometric numerical integration}, volume~31 of {\em Springer Series in Computational Mathematics}.
\newblock Springer, Heidelberg, 2010.
\newblock Structure-preserving algorithms for ordinary differential equations, Reprint of the second (2006) edition.

\bibitem{HWY18}
Kyle Helfrich, Devin Willmott, and Qiang Ye.
\newblock Orthogonal recurrent neural networks with scaled {C}ayley transform.
\newblock In Jennifer Dy and Andreas Krause, editors, {\em Proceedings of the 35th International Conference on Machine Learning}, volume~80 of {\em Proceedings of Machine Learning Research}, pages 1969--1978. PMLR, 10--15 Jul 2018.

\bibitem{Helgason78}
Sigurdur Helgason.
\newblock {\em Differential geometry, {L}ie groups, and symmetric spaces}, volume~80 of {\em Pure and Applied Mathematics}.
\newblock Academic Press, Inc. [Harcourt Brace Jovanovich, Publishers], New York-London, 1978.

\bibitem{Iserles01}
A.~Iserles.
\newblock On {C}ayley-transform methods for the discretization of {L}ie-group equations.
\newblock {\em Found. Comput. Math.}, 1(2):129--160, 2001.

\bibitem{JHD20}
Michael Jauch, Peter~D. Hoff, and David~B. Dunson.
\newblock Random orthogonal matrices and the {C}ayley transform.
\newblock {\em Bernoulli}, 26(2):1560--1586, 2020.

\bibitem{JH03}
Yindi Jing and Babak Hassibi.
\newblock Unitary space-time modulation via {C}ayley transform.
\newblock {\em IEEE Trans. Signal Process.}, 51(11):2891--2904, 2003.

\bibitem{JM17}
M.~T. Jury and R.~T.~W. Martin.
\newblock Non-commutative {C}lark measures for the free and abelian {T}oeplitz algebras.
\newblock {\em J. Math. Anal. Appl.}, 456(2):1062--1100, 2017.

\bibitem{Kai07}
Chifune Kai.
\newblock A characterization of symmetric {S}iegel domains by convexity of {C}ayley transform images.
\newblock {\em Tohoku Math. J. (2)}, 59(1):101--118, 2007.

\bibitem{Knapp02}
Anthony~W. Knapp.
\newblock {\em Lie groups beyond an introduction}, volume 140 of {\em Progress in Mathematics}.
\newblock Birkh\"auser Boston, Inc., Boston, MA, second edition, 2002.

\bibitem{KM03}
Bertram Kostant and Peter~W. Michor.
\newblock The generalized {C}ayley map from an algebraic group to its {L}ie algebra.
\newblock In {\em The orbit method in geometry and physics ({M}arseille, 2000)}, volume 213 of {\em Progr. Math.}, pages 259--296. Birkh\"{a}user Boston, Boston, MA, 2003.

\bibitem{Krantz99}
Steven~G. Krantz.
\newblock {\em Handbook of complex variables}.
\newblock Birkh\"auser Boston, Inc., Boston, MA, 1999.

\bibitem{PX01}
P.~S. Krishnaprasad and Xiaobo Tan.
\newblock Cayley transforms in micromagnetics.
\newblock {\em Physica B: Condensed Matter}, 306(1):195--199, 2001.
\newblock Proceedings of the Third International Symposium on Hysteresis an d Micromagnetics Modeling.

\bibitem{Lance95}
E.~C. Lance.
\newblock {\em Hilbert {$C^*$}-modules}, volume 210 of {\em London Mathematical Society Lecture Note Series}.
\newblock Cambridge University Press, Cambridge, 1995.
\newblock A toolkit for operator algebraists.

\bibitem{LPR06}
Nicole Lemire, Vladimir~L. Popov, and Zinovy Reichstein.
\newblock Cayley groups.
\newblock {\em J. Amer. Math. Soc.}, 19(4):921--967, 2006.

\bibitem{Domingo75}
Domingo Luna.
\newblock Letter to v.~l.~popov, 1975.

\bibitem{MR96}
Karl Meerbergen and Dirk Roose.
\newblock Matrix transformations for computing rightmost eigenvalues of large sparse non-symmetric eigenvalue problems.
\newblock {\em IMA J. Numer. Anal.}, 16(3):297--346, 1996.

\bibitem{MD24}
Omar Melikechi and David~B. Dunson.
\newblock Ellipsoid fitting with the {C}ayley transform.
\newblock {\em IEEE Trans. Signal Process.}, 72:70--83, 2024.

\bibitem{MST24}
Samir Mondal, K.~C. Sivakumar, and Michael Tsatsomeros.
\newblock The {C}ayley transform of prevalent matrix classes.
\newblock {\em Linear Algebra Appl.}, 681:1--20, 2024.

\bibitem{Nomura01}
Takaaki Nomura.
\newblock A characterization of symmetric {S}iegel domains through a {C}ayley transform.
\newblock {\em Transform. Groups}, 6(3):227--260, 2001.

\bibitem{Nomura03}
Takaaki Nomura.
\newblock Geometric norm equality related to the harmonicity of the {P}oisson kernel for homogeneous {S}iegel domains.
\newblock {\em J. Funct. Anal.}, 198(1):229--267, 2003.

\bibitem{onishchik2012lie}
A.~L. Onishchik and \`E.~B. Vinberg.
\newblock {\em Lie groups and algebraic groups}.
\newblock Springer Series in Soviet Mathematics. Springer-Verlag, Berlin, 1990.
\newblock Translated from the Russian and with a preface by D. A. Leites.

\bibitem{Postnikov86}
M.~Postnikov.
\newblock {\em Lie groups and {L}ie algebras}.
\newblock ``Mir'', Moscow, 1986.
\newblock Lectures in geometry. Semester V, Translated from the Russian by Vladimir Shokurov.

\bibitem{Quillen85}
Daniel Quillen.
\newblock Superconnections and the {C}hern character.
\newblock {\em Topology}, 24(1):89--95, 1985.

\bibitem{Quillen88}
Daniel Quillen.
\newblock Superconnection character forms and the {C}ayley transform.
\newblock {\em Topology}, 27(2):211--238, 1988.

\bibitem{ST20}
Yuliang Shen and Shuan Tang.
\newblock Weil-{P}etersson {T}eichm\"uller space {II}: {S}moothness of flow curves of {$H^{\frac 32}$}-vector fields.
\newblock {\em Adv. Math.}, 359:106891, 25, 2020.

\bibitem{SF70}
B\'ela Sz.-Nagy and Ciprian Foia\c~s.
\newblock {\em Harmonic analysis of operators on {H}ilbert space}.
\newblock North-Holland Publishing Co., Amsterdam-London; American Elsevier Publishing Co., Inc., New York; Akad\'emiai Kiad\'o, Budapest, 1970.
\newblock Translated from the French and revised.

\bibitem{vWW24}
Dragomir \v~Sari\'c, Yilin Wang, and Catherine Wolfram.
\newblock Circle {H}omeomorphisms with {S}quare {S}ummable {D}iamond {S}hears.
\newblock {\em Int. Math. Res. Not. IMRN}, (17):12219--12268, 2024.

\bibitem{wallach2017geometric}
Nolan~R. Wallach.
\newblock {\em Geometric invariant theory}.
\newblock Universitext. Springer, Cham, 2017.
\newblock Over the real and complex numbers.

\bibitem{WGM21}
Mich\`ele Wandelt, Michael G\"{u}nther, and Michelle Muniz.
\newblock Geometric integration on {L}ie groups using the {C}ayley transform with focus on lattice {QCD}.
\newblock {\em J. Comput. Appl. Math.}, 387:Paper No. 112495, 10, 2021.

\bibitem{chenweihuan2001}
Chen Weihuan.
\newblock {\em Introduction to Differentiable Manifold}.
\newblock Higher Education Press, second edition, 2001.

\bibitem{Weil60}
Andr\'e Weil.
\newblock Algebras with involutions and the classical groups.
\newblock {\em J. Indian Math. Soc. (N.S.)}, 24:589--623, 1960.

\bibitem{WY13}
Zaiwen Wen and Wotao Yin.
\newblock A feasible method for optimization with orthogonality constraints.
\newblock {\em Math. Program.}, 142(1-2, Ser. A):397--434, 2013.

\bibitem{Weyl39}
Hermann Weyl.
\newblock {\em The {C}lassical {G}roups. {T}heir {I}nvariants and {R}epresentations}.
\newblock Princeton University Press, Princeton, NJ, 1939.

\end{thebibliography}

 
\end{document}